\documentclass[11pt, a4paper]{article}
\usepackage[english]{babel}
\usepackage{amssymb,amsmath,amsthm,,graphicx}
\usepackage{subfigure}
\usepackage{epstopdf}
\usepackage{hyperref}
\usepackage{xcolor}
\usepackage[margin=2.15cm]{geometry}

\definecolor{MyDarkBlue}{rgb}{0,0.10,0.60}
\definecolor{BrickRed}{rgb}{0.65,0.08,0}

\hypersetup{
	colorlinks=true,       		
    linkcolor=MyDarkBlue,       
    citecolor=BrickRed,        	
    filecolor=red,      		
    urlcolor=cyan           	
}

\newcommand{\eqn}[1]{\begin{equation} #1 \end{equation}}
\newcommand{\eqan}[1]{\begin{align} #1 \end{align}}

\newtheorem{theorem}{Theorem}[section]
\newtheorem{lemma}[theorem]{Lemma}
\newtheorem{proposition}[theorem]{Proposition}

\theoremstyle{remark}
\newtheorem*{remark}{Remark}

\def \b {\beta}

\def \D {\Delta}
\def \e {\varepsilon}

\def \o {\omega}
\def \t {\tau}

\def \s {\sigma}
\def \g {\gamma}
\def \G {\Gamma}

\def \cF {\mathcal{F}}
\def \E {\mathbb E}
\def \ev {\mathcal E}
\def \ov {\mathcal O}
\def \N {\mathbb N}

\def \SS {\mathcal{S}}

\def \RE {\mathcal{R}_\ev}
\def \RO {\mathcal{R}_\ov}
\def \LL {\mathcal{L}}
\def \ninf {\nu \to \infty}

\def \liminfnu {\liminf_{\s \to \infty}}

\def \txy {T_{\ov \to \ev}}
\def \tyx {T_{\ev\to \ov}}
\def \uy {U_\ev}
\def \sy {S_\ev}
\def \stl {\geq_{{\rm st}}}
\def \ed {\,{\buildrel d \over =}\,}
\def \pap {\partial^+}
\def \pam {\partial^-}
\def \gs {L}
\newcommand{\indi}[1]{{\rm I}_{\{ #1 \}}}
\newcommand{\pr}[1]{\mathbb P \left( #1 \right)}

\definecolor{lg}{rgb}{0.6,0.6,0.6}

\def \odda {{\color{lg} \bullet}}
\def \oddi {\circ}
\def \evena {\bullet}
\def \eveni {{\color{lg} \circ}}

\begin{document}

\title{Delay performance in random-access grid networks}
\author{A. Zocca\thanks{Department of Mathematics and Computer Science, Eindhoven University of Technology, P.O. Box 513, 5600 MB Eindhoven, The Netherlands.} \and S.C. Borst\footnotemark[1] \thanks{Alcatel-Lucent Bell Labs, P.O. Box 636, Murray Hill, NJ 07974-0636, USA.} \and J.S.H. van Leeuwaarden\footnotemark[1] \and F.R. Nardi\footnotemark[1] }
\date{}

\maketitle

\begin{abstract}
We examine the impact of torpid mixing and meta-stability issues on the delay performance in wireless random-access networks.
Focusing on regular meshes as prototypical scenarios, we show that the mean delays in an $L\times L$ toric grid with normalized load~$\rho$ are of the order $(\frac{1}{1 - \rho})^L$. This superlinear delay scaling is to be contrasted with the usual linear growth of the order $\frac{1}{1 - \rho}$ in conventional queueing networks. The intuitive explanation for the poor delay characteristics is that (i) high load requires a high activity factor, (ii) a high activity factor implies extremely slow transitions between dominant activity states, and (iii) slow transitions cause starvation and hence excessively long queues and delays.

Our proof method combines both renewal and conductance arguments. A critical ingredient in quantifying the long transition times is the derivation of the communication height of the uniformized Markov chain associated with the activity process. We also discuss connections with Glauber dynamics, conductance and mixing times.
Our proof framework can be applied to other topologies as well, and is also relevant for the hard-core model in statistical physics and the sampling from independent sets using single-site update Markov chains.
\end{abstract}

\noindent \textbf{Keywords:} wireless random-access networks; grid topologies, delay scaling; hitting times; conductance; mixing times.

\section{Introduction}

Emerging wireless mesh networks typically lack any centralized control
entity, and instead vitally rely on a distributed mechanism for
regulating the access of the various nodes to the shared medium.
A popular mechanism for distributed medium access control is provided
by the so-called Carrier-Sense Multiple-Access (CSMA) protocol,
various incarnations of which are implemented in IEEE 802.11 networks.
In the CSMA protocol each node attempts to access the medium after
a random back-off time, but nodes that sense activity of interfering
nodes freeze their back-off timer until the medium is sensed idle.

While the CSMA protocol is fairly easy to understand at a local level,
the interaction among interfering nodes gives rise to quite intricate
behavior and complex throughput characteristics on a macroscopic scale.
In recent years relatively parsimonious models have emerged that
provide a useful tool in evaluating the throughput characteristics
of CSMA-like networks \cite{BKMS87,DDT07,DT06,WK05}.
These models essentially assume that the interference constraints can
be represented by a general conflict graph, and that the various nodes
activate asynchronously whenever none of their neighbors are presently
active.
Although such a representation of the IEEE 802.11 back-off mechanism
is not as detailed as in the landmark work of Bianchi~\cite{Bianchi00},
the general conflict graph provides far greater versatility
in describing a broad range of topologies.
Experimental results in Liew {\em et al.}~\cite{LKLW08} demonstrate
that these models, while idealized, yield throughput estimates that
match remarkably well with measurements in actual IEEE 802.11 networks.

Besides long-term throughput estimates, the above-described models also
provide valuable qualitative insights in transient characteristics.
In particular, the models reveal that the activity process of the
various nodes may exhibit slow mixing and meta-stable behavior due to
slow transitions between dominant states~\cite{LK11,ZBvL12}.
As a result, individual nodes will experience long sequences of
transmissions in rapid succession, interspersed with extended periods
of starvation.
Since the above-described models assume saturated buffers, they do not
explicitly capture the effect of temporal starvation on queue lengths
and delays though.
In the present paper we therefore augment the basic model with
queueing dynamics, and analyze average steady-state delays in order
to examine the performance repercussions of temporal starvation.

Only few results for average steady-state delays in random-access
networks have been obtained so far.
An interesting paper by Shah {\em et al.}~\cite{STT11} showed that
low-complexity schemes cannot achieve low delay in arbitrary topologies
(unless P equals NP), since that would imply that certain NP-hard
problems could be solved efficiently.
However, the notion of delay in~\cite{STT11} is a transient one,
and it is not exactly clear what the implications are for the
average steady-state delay in specific networks, if any.
Jiang {\em et al.}~\cite{JLNSW11} derived {\em upper\/} bounds
for the average steady-state delay.
The bounds show that for low load the delay only grows polynomially
with the number of nodes in bounded-degree graphs.
Subramanian \& Alanyali~\cite{SA11} presented similar upper bounds for
bounded-degree graphs with low load based on analysis of neighbor sets
and stochastic coupling arguments.

Focusing on regular meshes as prototypical scenarios, we establish
that the average steady-state delays in an $L\times L$ toric grid with
normalized load~$\rho$ scale as $(\frac{1}{1 - \rho})^L$,
in contrast to the typical linear growth of the order $\frac{1}{1 - \rho}$ in
conventional queueing networks.
The intuitive explanation for the adverse delay scaling is that
(i) high load requires a high activity factor,
(ii) a high activity factor implies extremely slow transitions between
dominant activity states,
and (iii) slow transitions cause starvation and hence excessively long
queues and delays.

Our novel proof method combines both renewal and conductance arguments.
A critical role in determining the long transition times in (ii) is
played by the derivation of the communication height of the
uniformized Markov chain associated with the activity process.
As a side-result, we gain valuable insight in the exponentially small
conductance of Glauber dynamics on the independent sets in the $L \times L$
grid, and hence in the resulting exponentially large mixing time.
Our proof methodology can be applied to other graphs as well, and is also
relevant for the hard-core model in statistical physics~\cite{vdBS94}
and the sampling from independent sets using single-site update Markov
chains in the theoretical computer science literature~\cite{B99,G08,R06}.

The remainder of the paper is organized as follows.
In Section~\ref{mode} we present a detailed model description.
In Section~\ref{overview} we give an overview of the main results
for the delays, transition times, communication height and mixing times, 
and provide heuristic interpretations and high-level sketches of the
proof arguments.
In Section~\ref{sec:proofrenewal} we establish the intimate connection
between long transition times and long delays.
Section~\ref{sec:ch} is devoted to the derivation of the communication
height, which plays a pivotal role in characterizing the slow
transitions between dominant activity states, proved in Section~\ref{sec:thm31}. 
In Section~\ref{concl} we make some concluding remarks and sketch some
directions for further research. 

\section{Model description}
\label{mode}

In this section we present a detailed description of the model for
random-access networks with queueing dynamics.
In Subsection~\ref{general} we specify the network dynamics for
arbitrary conflict graphs with node-specific arrival, transmission,
and activation rates.
To facilitate a more insightful analysis, we introduce
in Subsection~\ref{sec:sym} several symmetry assumptions, amounting to
toric grids with homogeneous arrival, transmission, and activation rates.
This symmetric network will be the focus throughout the paper.
In Subsection~\ref{sub:stabcond} we establish the stability conditions for the network.

\subsection{Network dynamics}
\label{general}

Consider a network of $N$~nodes sharing a wireless medium.
The network is represented by an undirected graph $G = (V, E)$,
called \textit{conflict graph} or \textit{interference graph},
where the set of vertices $V = \{1,\ldots, N\}$ correspond to the
various nodes and the set of edges $E \subseteq V \times V$ indicate
which pairs of nodes interfere.
Nodes that are neighbors in the conflict graph are prevented from
simultaneous activity, and thus the independent sets of~$G$ (sets of
vertices not sharing any edge) correspond to the feasible joint
activity states of the network.
A node is said to be \textit{blocked} whenever the node itself or any
of its neighbors is active, and \textit{unblocked} otherwise. 
Denote by $\Omega \subseteq \{0, 1\}^N$ the collection of all
independent sets of~$G$. Packets arrive at node~$i$ as a Poisson process of rate~$\lambda_i$.
The packet transmission times at node~$i$ are independent
and exponentially distributed with mean $1 / \mu_i$. Let $X(t) \in \Omega$ represent the joint activity state of the network at time~$t$, with $X_i(t)$ indicating whether node~$i$ is active at time~$t$ or not. Let $\{Q(t)\}_{t \geq 0}$ be the joint queue-length process, with $Q_i(t)$ the number of packets waiting at node~$i$ at time~$t$ (excluding the packet that may be in the process of being transmitted).

The various nodes share the medium according to a random-access mechanism.
When a node ends a packet transmission, it either starts a back-off
period with probability~$p_i$ before the next packet transmission,
or immediately starts the next packet transmission otherwise.
The back-off periods of node $i$ are independent and exponentially
distributed with mean $1/\nu_i$.
The back-off period of a node is suspended whenever it becomes blocked
by activity of any of its neighbors, and only resumes once the node
becomes unblocked again.

Under this random-access mechanism, the process
$\{(X(t), Q(t))\}_{t \geq 0}$ evolves as a continuous-time Markov
process with state space $\Omega \times {\mathbb N}^N$.
Transitions from a state $(X, Q)$ to $(X, Q + e_i)$ due to arrivals
occur at rate~$\lambda_i$, transitions due to activations from
a state $(X, Q)$ with $X_i = 0$ and $X_j = 0$ for all neighbors~$j$
of node~$i$ to $(X + e_i, Q - e_i \indi{Q_i > 0})$ occur at rate~$\nu_i$,
transitions due to transmission completions followed by a back-off
period from a state $(X, Q)$ with $X_i = 1$ to $(X - e_i, Q)$ occur
at rate $p_i \mu_i$, and transitions due to transmission completions
immediately followed by another transmission (without an intermediate
back-off period) from a state $(X, Q)$ with $X_i = 1$ to
$(X, Q - e_i \indi{Q_i > 0})$ occur at rate $(1 - p_i) \mu_i$.

For any $x \in \Omega$, define
$\pi(x) = \lim_{t \to \infty} \pr{X(t) = x}$ as the stationary
probability that the activity process resides in state~$x$.
The activity process $\{X(t)\}_{t \geq 0}$ does not depend on the
process $\{Q(t)\}_{t \geq 0}$, and in fact constitutes a reversible
Markov process with a product-form stationary distribution~\cite{BKMS87}
\[
\pi(x) = Z^{-1} \prod_{i=1}^N \s_i^{x_i}, \quad x \in \Omega,
\]
where $\s_i = \nu_i / (p_i \mu_i)$ and $Z$ is the normalization constant
defined as
\[
Z=\sum_{x \in \Omega} \prod_{i=1}^N \s_i^{x_i}.
\]

In contrast, the queue-length process $\{Q(t)\}_{t \geq 0}$ \textit{does}
strongly depend on the activity process $\{X(t)\}_{t \geq 0}$,
and is considerably harder to analyze.
Since the evolution of $\{Q(t)\}_{t \geq 0}$ is modulated by that of $\{X(t)\}_{t \geq 0}$, the former process can be viewed as a queueing
network in a random environment, and in principle the queue-length process
at each individual node could be considered as a quasi-birth-and-death
process, and analyzed using matrix-analytic techniques.
However, this would only yield numerical solutions and not provide
explicit insight in how the queue-length process and delay performance
are affected by the temporal starvation issues.

\begin{remark}
We have implicitly assumed that a node may become active even when its buffer is empty, and then transmits dummy packets. It would be interesting to extend the analysis to the case where nodes with empty buffers refrain from transmission activity, but challenging since the behavior of the activity process then does depend on the queue-length process.
\end{remark}

\subsection{Symmetry assumptions}
\label{sec:sym}

To obtain more transparent results, we henceforth focus on a symmetric
scenario with homogeneous arrival, transmission and activation rates,
as well as equal back-off probabilities at all nodes:
\[
\lambda_i \equiv \lambda, \quad \mu_i \equiv \mu, \quad \nu_i \equiv \nu, \quad p_i \equiv p, \quad i = 1, \dots, N.
\]
The stationary distribution of the activity process then simplifies to
\eqn{
\label{eq:pih}
\pi(x) = Z^{-1} \s^{\sum_{i=1}^{N} x_i}, \quad x \in \Omega,
}
where $\s = \nu / (p \mu)$ and the normalization constant may be written as
$Z = \sum_{x \in \Omega} \s^{\sum_{i=1}^{N} x_i}$.
As~\eqref{eq:pih} shows, when $\s > 1$, the stationary probability of
an activity state $x \in \Omega$ increases with its cardinality in an exponential fashion.

As prototypical examples of a symmetric topology, we will focus on toric
grids, i.e.~with a periodic (wrap-around) boundary to preserve symmetry.
More precisely, we consider as conflict graph $G=(V,E)$ the even discrete torus of size $\gs$. The vertex set is $V=\{1,\dots,\gs\}^2$, hence $N=|V|=\gs^2$, and two vertices in $V$ are adjacent if they differ by $1 \pmod \gs$ in only one coordinate. A vertex is called \textit{even} (\textit{odd}) if the sum of its two coordinates is even (odd). The vertex set $V$ is thus partitioned into two classes: the set $\ev$ of even nodes and the set $\ov$ of odd nodes. Clearly $|\ev|=\gs^2/2=|\ov|$ and these are the two maximum independent sets of the graph $G$.
In the case of the toric grid, \eqref{eq:pih} shows that, as $\s$ grows large, the probability mass will concentrate on the maximum independent sets $\ev$ and $\ov$, to which we will also refer as \textit{dominant activity states}.

For the toric grid the conflict graph $G$ is a bipartite graph, and in the special case $\gs=2$, it becomes even a \textit{complete} bipartite graph with components $\ev$ and $\ov$. This conflict graph has additional features which simplify its analysis~\cite{ZBvL12}, and in the remainder of the paper we will therefore assume that $\gs\geq 4$.

By virtue of the symmetry, all the nodes will be active the same
fraction of time~$\theta$, which will approach 1/2 as $\s$ grows large,
since the probability mass will then concentrate on the two dominant
activity states $\ev$ and $\ov$.
However, the transitions between the two dominant activity states
will occur at an increasingly slow rate, resembling meta-stability
phenomena in statistical physics~\cite{MNOS04,OV05}.
As a result, individual nodes will experience long sequences of
transmissions in rapid succession, during which their buffers drain,
interspersed with extended periods of starvation, during which large numbers
of packets accumulate.
Our objective in this paper is to examine the effect
of the temporal starvation on queue-lengths and delays.
As it turns out, the analysis will revolve around the
\textit{transition times} $\tyx$ and $\txy$ between the two dominant
activity states defined as
\[
\tyx:=\inf \{ t >0 : X(t)=\ov ~|~ X(0)=\ev\} \quad \text{ and } \quad \txy:=\inf \{ t >0 : X(t)=\ev ~|~ X(0)=\ov\}.
\]

Note that these two random variables only depend on the activity process $\{X(t)\}_{t \geq 0}$, and are independent of the queue-length process $\{Q(t)\}_{t \geq 0}$.

\subsection{Stability conditions}
\label{sub:stabcond}
In view of the above observations it is natural to define the
normalized load as $\rho = \lambda / 2 \mu$ and to assume $\rho < 1$,
since otherwise all the queues will be unstable.
The latter condition is not only necessary for stability, but in fact
also sufficient for the queue-length processes at each individual node
to be positive recurrent for large enough values of~$\s$, as stated in
the next proposition.

\begin{proposition}
\label{prop:posrec}
Assume that $\rho < 1$.
Then the queue-length process $\{Q_i(t)\}_{t \geq 0}$ at each individual
node is positive recurrent for sufficiently large values of~$\s$.
\end{proposition}

The above proposition may be intuitively explained as follows.
As mentioned earlier, the fraction of time~$\theta$ that each
individual node is active will approach 1/2 arbitrarily close for
sufficiently large values of~$\s$.
This suggests that $\lambda < \mu / 2$, i.e.~$\rho < 1$,
is sufficient for stability for large enough values of~$\s$.

\begin{proof}
We say that the activity process resides in an even period if the last dominant state it visited was $\ev$ and likewise in an odd period if the last dominant state it visited was $\ov$. We refer to a period consisting of an even and subsequent odd period as a cycle. Define $\{\hat{Q}_n\}_{n \in \N}$ to be the queue-length process of an odd node embedded at the epochs when even periods start (or equivalently when cycles start). Thanks to the strong Markov property, this latter process is an irreducible time-homogeneous Markov chain on $\N$. Denote by $S_n$ the number of packets served during the $n$-th odd period and by $A_n$ the total number of packets which arrived during the $n$-th cycle. By construction these two random variables are independent of each other and independent from those of different cycles. We claim that the following inequality holds:
\[
\hat{Q}_{n+1} \leq (\hat{Q}_n-S_n )^+ +A_n.
\]
Indeed we underestimate the total number of packets served by allowing transmission only during the odd period (instead of continuously) and only the awaiting packets which arrived in the previous cycles.

Consider now another fictitious Markov chain $\{\bar{Q}_n\}_{n \in \N}$ defined by the initial condition $\bar Q(0) = \hat Q(0)$ and by the classical Lindley recursion $\bar Q_{n+1}=(\bar Q_n-S_n+A_n)^+$. Clearly $\hat Q_n \leq \bar Q_n+A_n$ for every $n \in \N$.
Let $Y$ be a random variable distributed as $A-S$. It is well known that $0$ is a positive recurrent state for $\{\bar Q_n\}_{n \in \N}$ if $\E Y < 0$. This is indeed the case, since
$\E Y= \lambda \E \tyx  + (\lambda-\mu) \E \txy  = (2 \lambda - \mu) \E \tyx = \mu (\rho-1) <0.$
Therefore also for the chain $\{\hat Q_n\}_{n\in\N}$ the state $0$ is positive recurrent and, since the cycle time has finite expectation (the Markov process $\{X(t)\}_{t\geq 0}$ is irreducible on a finite state space, hence positive recurrent), it follows that the original continuous-time process $\{Q_i(t)\}_{t\geq 0}$ is positive recurrent as well.
\end{proof}

The next proposition provides a lower bound for the value of~$\s$
that is required for stability in terms of the normalized load~$\rho$.

\begin{proposition}
\label{prop:stable}
In order for all queues to be stable, it is required that
\[\s > \frac{\rho}{2 (1-\rho)}.\]
\end{proposition}

\begin{proof}
In order to be stable, each individual node should be active at least
a fraction $\lambda/\mu=\rho/2$ of the time, and thus $\theta > \rho/2$.
The fraction of time that a node is active may be expressed as
$\theta = \sigma \pr{\text{node is unblocked}}$.
Now consider an arbitrary node and any of its neighbors.
Then the fraction of time that the former node is unblocked is bounded
from above by the fraction of time that both nodes are inactive, so that
$\pr{\text{node is unblocked}} \leq 1 - \rho$.
Combining the above three inequalities and some rearranging yields
the result.
\end{proof}

\section{Main results}
\label{overview}

In Subsection~\ref{sub1} we present the two main theorems, which lead
to bounds for the long-term average delay in random-access grid networks. 
The first result follows from a third theorem on the so-called
communication height, presented in Subsection~\ref{sub2}, along with
a sketch of the proof.
In Subsection~\ref{sub3} we show how the result for the communication
height can be leveraged to characterize the mixing time of the
activity process, and we also discuss connections with other results
in the areas of statistical physics and theoretical computer science. 

\subsection{Transition times and delays}
\label{sub1}

As mentioned earlier, the bounds we derive revolve around three simple
observations:
(i) {high load} requires {high activity factors} for stability;
(ii) {high activity factors} cause {long mixing times}, in particular
{\em slow transitions between dominant activity states};
(iii) {slow transitions between dominant states} imply {long starvation
periods}, and hence {large queue-lengths and delays}.

The next two theorems formalize the last two observations.
We henceforth assume $\mu$ to be fixed, and when we write $\s \to \infty$ we allow for either $\ninf$, $p \downarrow 0$, or both. For compactness we only attach $\s$ in brackets to various quantities to reflect the dependence on both $\nu$ and $p$ in limit statements for $\s \to \infty$. Define $\alpha := \liminfnu \frac{\log p}{\log p - \log \nu} \geq 0$.

\begin{theorem}[Transition time between dominant states]\label{thm1}
\[
\liminfnu \frac{\log \E \tyx(\s) }{\log \s}\geq \gs + \alpha \geq L.
\]
\end{theorem}

With minor abuse of notation, let $\E W(\s)$ be the long-term average delay (waiting time plus service time) of an arbitrary packet, where again the attached $\s$ in brackets describes the dependence on both $\nu$ and $p$.

\begin{theorem}[Long-term average delay]\label{thm2}
\[
\liminfnu \frac{\E W(\s)}{\E \tyx(\s)} \geq \frac{1}{4} \frac{\mu}{\mu - \lambda} = \frac{1}{4 - 2 \rho}.
\]
\end{theorem}

In view of Proposition~\ref{prop:stable} and Theorem~\ref{thm2}, an immediate corollary of the above theorem is that the long-term average delay must scale as $(\frac{1}{1 - \rho})^{\gs}$ as $\rho \uparrow 1$. This is to be contrasted with the usual linear scaling in $\frac{1}{1 - \rho}$ in conventional queueing networks, and further reveals that the exponential rate of growth increases with the size of the grid.

Theorem~\ref{thm2} may be understood as follows. First, by virtue of the symmetry, we may as well consider the average delay of an arbitrary packet at a particular node, say an odd one. In view of the nature of the transitions between the dominant states~$\ev$ and~$\ov$, this node will hardly be active during the
transition time $\tyx(\s)$ for large values of~$\sigma$.
As a result, the queue at this node will roughly grow at rate~$\lambda$, reaching an expected length of $\lambda \E \tyx(\s)$ at the end of the transition period, and thus having an average size of approximately $\frac{1}{2} \lambda \E \tyx(\s)$ over this period.
During the subsequent transition time $\txy(\s)$, the queue will decrease at a rate of at most $\lambda - \mu$, and thus requires a time period of at least $\frac{\lambda}{\mu - \lambda} \E \tyx(\s)$ to empty, so that the average size over that period is at least 
\[
\frac{1}{2} \frac{\lambda}{\mu - \lambda} \frac{\E ((\tyx(\s))^2)}{\E \tyx(\s)}.
\]
Noting that $\E ((\txy(\s))^2) \geq (\E \txy(\s))^2$ and $\E \txy(\s) = \E \tyx(\s)$, we conclude that the average queue-length over the course of the cycle consisting of the transition periods $\tyx(\s)$ and $\txy(\s)$ is at least
\[
\frac{1}{2} \left ( \frac{1}{2} \lambda +
\frac{1}{2} \frac{\lambda}{\mu - \lambda}\right) \E \tyx(\s) =
\frac{\lambda}{4} \frac{\mu}{\mu - \lambda} \E \tyx(\s) =
\frac{\lambda}{4 - 2 \rho} \E \tyx(\s)
\]
for sufficiently large values of~$\sigma$. Invoking Little's law, we then obtain the lower bound as stated in the theorem for the long-term average delay of an arbitrary packet.

Of course, the above arguments are heuristic in nature. We will provide a rigorous proof of Theorem~\ref{thm2} in Section~\ref{sec:proofrenewal}. Specifically, we will establish that an odd node receives hardly any service during a transition time $\tyx(\s)$ for large values of~$\s$, and justify the informal calculations based on drifts and expected values.

\subsection{Communication height}
\label{sub2}

The high-level idea of the proof of Theorem~\ref{thm1} is that, in order for a transition from the even dominant state $\ev$ to the odd dominant state $\ov$ to occur, the activity process $X(t)$ must follow a transition path through some highly unlikely configurations consisting of mixtures of even and odd active nodes, and the time to reach such configurations is correspondingly long. Theorem~\ref{thm3} below shows a key feature of all such transition paths. 

In the rest of the paper we will interchangeably refer to a feasible configuration $x \in \Omega$ or to the corresponding independent set $I$ of the conflict graph $G$. Clearly $\sum_{i=1}^N x_i = |I|$. For any independent set $I\in \Omega$ define $\D(I)$ to be the difference between the size of a maximum independent set and that of the set $I$, i.e. $\D(I):=\gs^2/2-|I|$. The quantity $\D(I)$ describes how ``inefficient'' the configuration $I$ is with respect to the maximum independent sets $\ev$ and $\ov$, and hence how much less likely it is in view of~\eqref{eq:pih}. For this reason, we will refer to $\D(I)$ as the \textit{efficiency gap} of configuration $I$. A {\it path} is a sequence $\o=(I_1,\dots,I_n)$, $n\in\N$, where $I_k\in\Omega$ for $k=1,\dots,n$ and such that $I_k$ and $I_{k+1}$ differ by a single flip, for $k=1,\dots,n-1$. We write $\o: I \to I'$ to denote a path from $I$ to $I'$. The {\it communication height\/} between two states $I,I'\in\Omega$ is
\[
\phi(I,I') := \min_{\o : I\to I'} \max_{J \in\o} \D(J).
\]
The quantity $\phi(I,I')$ is the minimum number of nodes which must be simultaneously inactive with respect to the maximum number of active nodes to make the transition $I \to I'$ possible. The cornerstone of the proof of Theorem~\ref{thm1} is the derivation of the communication height $\G:=\phi(\ev,\ov)$ between the two dominant configurations $\ev$ and $\ov$.
\begin{theorem}[Communication height]\label{thm3}
\[
\Gamma=\gs+1.
\]
\end{theorem}

The proof of Theorem~\ref{thm3}, given in Section~\ref{sec:ch}, is technically challenging, since all possible paths between the dominant states need to be taken into consideration. In order to do so, we use a similar approach as in~\cite{B99, G08, R06} which associates with each independent set $I \in \Omega$ a collection of cutsets and contours, that effectively translates the problem into geometric properties of these contours.

Theorem~\ref{thm3} is a powerful result. In Section~\ref{sec:thm31} we show
that together with a general result for exit times for interacting particle systems
in statistical physics~\cite{OV05} it readily leads to Theorem~\ref{thm1}.
Theorem~\ref{thm3} also gives a way to derive results for mixing times. Since the majority of related literature has been dealing with mixing times, instead of transition times, we present in Subsection~\ref{sub3} the mixing time result that can be obtained using Theorem~\ref{thm3}.

\subsection{Mixing times}
\label{sub3}

Let $d(t)$ be the maximal distance over $x\in\Omega$, measured in terms of total variation, between the distribution at time~$t$ and the stationary distribution, i.e. 
\[
d(t):=\max_{x\in\Omega}\|\pr{X(t) \in \cdot ~|~ X(0)=x}-\pi\|_{\rm TV},
\]
where $\| p-q \|_{\rm TV}=\frac{1}{2} \sum_{x \in \Omega} |p(x)-q(x)|$ is the total variation distance between two probability measures $p,q$ on the finite state space $\Omega$.
The mixing time of the process $\{X(t)\}_{t\geq 0}$ is defined as
\[
t_{{\rm mix}}(\e,\s)=\inf\{t \geq 0 : d(t)\leq \e \}
\]
and expresses the time needed for a process to converge to stationarity. 

An important observation is that when $\sigma$ is large, the
distribution in~\eqref{eq:pih} favors dense configurations with
$\D(I)$ small and with extremes $\ev$ and $\ov$, while configurations
that are roughly half odd and half even are highly unlikely.
This `bimodality' causes slow convergence to equilibrium,
because the Markov process is either in the even or odd sublattice,
while it can hardly pass through the less likely set in the middle
(the bottleneck) to go from the even to the odd sublattice or vice versa.
Hence, a classical strategy for proving slow mixing is to separate the
state space into three sets: Two of which are exponentially larger
(in terms of probability) than the third, such that moving from one
large set to the other through local steps of the Markov process
requires passing through the small bottleneck set in the middle.  
This strategy has been applied in \cite{B99,G08,R06}
to prove slow mixing of local or single-site update Markov processes
(like ours) on the grid.
Traditionally, the two dominant sets are those configurations that lie predominantly on the even or odd sublattice, and the bottleneck set consists of those configurations that have a relatively balanced number of even and odd vertices. The challenge is then to show that this bottleneck or balanced set has exponentially small probability, in which case the conductance of the Markov process would be exponentially small, leading via a standard argument to exponentially large mixing times. This challenge essentially boils down to counting the number of balanced configurations, which is solved in \cite{B99,G08} by powerful combinatorial enumeration methods. In \cite{R06} a different way of partitioning was introduced, that divides the state space according to a certain topological obstruction, rather than according to the relative number of even and odd vertices.
In that case, again  a counting problem needs to be solved, in order to establish that the set of configurations having this topological obstruction is exponentially small, but the counting problem is much simpler. Recently, the same approach was followed in \cite{BGTR12}, but with sharper estimates for the counting problem.

What is new about our approach is that we partition the state space in a fundamentally different way, such that there is no longer the need to solve a counting problem. The approach is best explained in terms of the set
\eqn{\label{eq:defS}
\SS:=\{ I \in \Omega ~|~ \phi(\ev,I) \leq \gs \}.
}
From the perspective of the maximum configuration $\ev$, the set $\SS$ contains all configurations that can be reached by the process by having at most $\gs$ nodes simultaneously inactive with respect to the maximum number of active nodes. Moreover, Theorem~\ref{thm3} tells us that the Markov process starting in $\ev$ cannot reach the configuration $\ov$ without leaving the set $\SS$. Hence, the two dominant sets are $\SS=\{ I \in \Omega ~|~ \phi(\ev,I) \leq \gs \}$, and by symmetry $\SS':=\{ I \in \Omega ~|~ \phi(\ov,I) \leq \gs \}$, while the remaining configurations together form the bottleneck. Now, since the distribution~\eqref{eq:pih} readily provides an upper bound on the probability of this bottleneck set, there is no counting problem involved in proving exponentially small conductance and hence slow mixing; see Theorem~\ref{thm:mix} below.

Let us now demonstrate the above reasoning to get an upper bound on the conductance of the activity process and thus a lower bound for the mixing time. These two results are not necessary for deriving the delay bounds, but are of independent interest in view of references~\cite{B99,G08,R06}. Therefore, for simplicity, we assume $p=1$ and $\mu=1$ in the rest of this subsection.

For $S \subseteq \Omega$, let $\pi(S):=\sum_{x \in S } \pi(x)$ be its stationary probability and define the \textit{conductance} of $S \subseteq \Omega$ as
\[\Phi(S):=\frac{\sum_{x \in \SS} \sum_{y \in \SS^c} \pi(x) q(x,y)}{\pi(S)},\]
where $q(x,y)$ is the transition rate from state $x$ to state $y$. The conductance of a Markov process, also known in literature as \textit{bottleneck ratio} or \textit{Cheeger constant}, is effectively a measure of how well connected the state space is, taking into account possible geometric features usually referred to as bottlenecks, that strongly influence the mixing time.

Let $\pam \SS$ denote the inner boundary of the set $\SS$, i.e. the states in $\SS$ that communicate with $\SS^c$.
\begin{lemma}[Upper bound on conductance of $\SS$]
\label{lem:cond}
\[
\Phi(\SS) \leq \frac{|\partial^- \SS| (\gs^2/2-\gs)}{\s^{\gs} }.
\]
\end{lemma}
\begin{proof}
From the definition of $\SS$ it follows that all states $x \in \pam \SS$ have the same stationary probability, namely $\pi(x)=Z^{-1} \s^{\gs^2/2-\gs}$. Moreover, the only way to exit from $\SS$ is by deactivating one node and this happens at rate $q(x, x')=\mu=1$, if $x\in \SS$ and $y \in \SS^c$. Hence
\[
\sum_{x \in \SS} \sum_{y \in \SS^c} \pi(x) q(x,y)= Z^{-1} \s^{\gs^2/2-\gs} \sum_{x \in \SS} \sum_{y \in \SS^c} 1 \leq Z^{-1} \s^{\gs^2/2-\gs} |\pam \SS| (\gs^2/2-\gs).
\]
The last inequality comes from the fact that the exit from $\SS$ can occur only by deactivation one node and there are exactly $\gs^2/2-\gs$ of them to be turned off in each configuration in $\pam \SS$. Trivially, $\pi(\SS) \geq \pi(\ev)=Z^{-1} \s^{\gs^2/2}$ and thus
\[
\Phi(\SS) \leq \dfrac{\s^{\gs^2/2-\gs}}{\s^{\gs^2/2}} |\partial^- \SS| (\gs^2/2-\gs) =\dfrac{|\partial^- \SS| (\gs^2/2-\gs)}{\s^{\gs}},
\]
which completes the proof.
\end{proof}

Due to the relation (see~\cite{ZBvL12})
\[
t_{{\rm mix}}(\epsilon) \geq \left(\frac{1}{2}-2\varepsilon\right) \max_{A \subseteq \Omega : \pi(A) \leq 1/2} \frac{1}{\Phi(A)},
\]
the following result is then an immediate consequence of Theorem~\ref{thm3} and Lemma~\ref{lem:cond}.

\begin{theorem}[Lower bound on mixing time]\label{thm:mix}
For any $\s>1$ the mixing time of the process $\{X(t)\}_{t\geq 0}$ satisfies
\[
t_{{\rm mix}}(\e,\s) \geq \left(\frac{1}{2}-2\epsilon\right) \frac{\s^{\gs}}{|\pam \SS|(\gs^2/2-\gs)}.
\]
\end{theorem}

\begin{remark}
At first sight, our novel approach might look deceptively simple, because no sophisticated enumeration method is needed. In fact, there is no counting problem to be solved. However, the technical challenge is shifted to establishing the communication height in Theorem~\ref{thm3}. Indeed, this result is very powerful, and the proof relies on an approach similar to that in~\cite{B99,G08,R06}. Once the communication height is established, Theorems~\ref{thm1} and~\ref{thm:mix} readily follow. 
\end{remark}

\section{Proof of Theorem~\ref{thm2} (long-term average delay)} 
\label{sec:proofrenewal}

In this section we present the proof of Theorem~\ref{thm2}.
Guided by the intuitive arguments in Subsection~\ref{sub1},
we consider an alternating renewal process. As before, an even renewal
period starts each time we observe a first entrance into the even
dominant state~$\ev$ after a visit to the odd dominant state~$\ov$,
and likewise an odd renewal period starts each time the process enters
the odd dominant state~$\ov$ for the first time after a visit to the
even dominant state~$\ev$.
Thus the lengths of the even and odd renewal periods correspond to the
transition times $\tyx$ and $\txy$, respectively.
We have defined a {\it cycle} as the period consisting of an even and subsequent odd renewal period.
Let $\sy$ and $\uy$ be two random variables representing the amounts
of time during an even renewal period that the activity process
$\{X(t)\}_{t \geq 0}$ spends in the even dominant state~$\ev$ and in other
activity states, respectively.
Since the random variable $\sy$ tracks the total time the process spends in the state $\ev$ in an entire cycle, the renewal-reward theorem yields
\[
\frac{\E \sy}{\E \tyx + \E \txy} = \pi(\ev) \uparrow \frac{1}{2}, \quad \s \to \infty,
\]
where we suppress the implicit dependence on~$\s$ for notational
compactness.
By symmetry, $\E \tyx = \E \txy$, so that
\[
\frac{\E \sy}{\E \tyx} \uparrow 1 \quad \text{ and } \quad \frac{\E \uy}{\E \tyx} \downarrow 0, \quad \s \to \infty.
\]
In particular, for any $\delta > 0$, there exists $\s_\delta > 0$
such that for all $\s > \s_\delta$
\eqn{\label{eq:Udelta}
\frac{\E \uy }{\E \tyx} \leq \delta.
}

As mentioned earlier, by virtue of the symmetry, we may consider the
long-term average delay of an arbitrary packet at a particular node,
say an odd one.
With minor abuse of notation, let
\eqn{\label{eq:lave}
\E L = \lim_{T \to \infty} \frac{1}{T} \int_{t = 0}^{T} L(t) {\rm d} t,
}
with $L(t)$ denoting the queue-length at that particular node at time~$t$,
i.e. the total number of packets, including the packet that may
possibly be in the process of being transmitted.
Assuming that $\E W$ exists, Little's law can then be invoked to
conclude that $\E L = \lambda \E W$ exists as well, so it suffices to
establish a lower bound for $\E L$.

By definition, $L(t)$ increases whenever a packet arrives at the node,
as governed by a Poisson process of rate~$\lambda$.
Further observe that $L(t)$ cannot decrease when the activity process
resides in the even dominant state $\ev$, since that precludes activity
of any odd node.
When the activity process does not reside in the even dominant state,
the node could potentially be active, and $L(t)$ decreases whenever
a transmission is completed and there are packets in the node ($L(t)\neq 0$).
Unfortunately, there are intricate dependencies that arise between
the transmission periods of the node under consideration and the
dynamics of the activity process, and in particular the durations
of time that it does not reside in the even dominant state.
However, the number of transmissions that occur during such a period
is a stopping time, and hence the expected number is bounded from
above by the amount of time divided by the mean transmission time.
Based on the above considerations, we conclude, conditioning on the durations of the even and odd renewal periods, that $\E_{\lambda, \mu} L(t)$ is bounded from below by $Z(t)$,
where the subscripts $\lambda$ and $\mu$ indicate expectation with
respect to the arrival epochs and transmission periods,
and $\{Z(t)\}_{t \geq 0}$ is a process with $Z(0) = L(0)$
that increases at rate~$\lambda$ when the activity process resides in
the even dominant state and decreases at rate $\lambda - \mu$ at all
other times, as long as $Z(t)$ is positive.
Unconditioning, this implies in particular that
\eqn{\label{eq:ineq}
\lim_{T \to \infty} \frac{1}{T} \int_{t = 0}^{T} \E L(t) {\rm d} t \geq \lim_{T \to \infty} \frac{1}{T} \int_{t = 0}^{T} \E Z(t) {\rm d} t.
}

In order to evaluate the latter limit expression,
let $T_k$ and $V_k$ be the durations of the $k$-th even and $k$-th odd
renewal periods, respectively, let $U_k$ be the amount of time that the
activity process does not reside in the even dominant state during the
$k$-th even renewal period, and let $S_K = \sum_{k = 1}^{K} (T_k + V_k)$
be the duration of the first $K$~cycles.
Note that $T_k$ and $V_k$ are i.i.d.~copies of the random variables
$\tyx$ and $\txy$, respectively.
Assume that an even renewal period starts at time~0 with $Z(0) = 0$,
and define the random variable $M = \inf\{K \geq 1: Z(S_K) = 0\}$.
Observe that $S_M$ is a regeneration epoch for the process
$\{Z(t)\}_{t \geq 0}$, and hence the renewal-reward theorem implies that
\eqn{\label{eq:zave}
\lim_{T \to \infty} \frac{1}{T} \int_{t = 0}^{T} Z(t) {\rm d} t = \frac{\E \int_{t = 0}^{S_M} Z(t) {\rm d} t}{\E S_M}.
}

Considering the denominator, it follows from Wald's equation and symmetry
considerations that $\E S_M = \E M (\E \tyx+ \E \txy) = 2 \E M \E \tyx$.
Turning attention to the numerator, we first condition on the number
of cycles~$M$ that have elapsed by the regeneration epoch, the
durations $T_1, T_2, \dots, T_M$ of the even renewal periods involved,
and the amounts of time $U_1, U_2, \dots, U_M$ that the activity
process does not reside in the even dominant state.

\begin{figure}[!ht]
\centering
\includegraphics[scale=0.22]{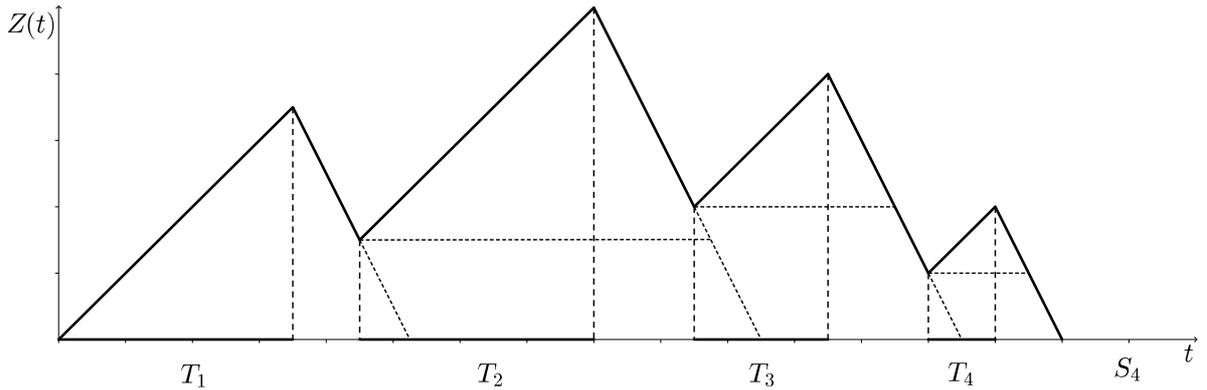}
\caption{Sample path of the process $Z(t)$, with $M=4$.}
\label{fig:zpath}
\end{figure}

When $U_k = 0$, inspection of Figure~\ref{fig:zpath} shows that the area of the
triangle associated with the $k$-th cycle is
\eqn{\nonumber
\frac{1}{2} \lambda T_k^2 +
\frac{1}{2} \frac{\lambda^2}{\mu - \lambda} T_k^2 =
\frac{1}{2} \frac{\lambda \mu}{\mu - \lambda} T_k^2,
}
while the area of the parallelogram associated with the $k$-th cycle is
\eqn{\nonumber
Z_k \Big(T_k + \frac{\lambda}{\mu - \lambda} T_k\Big) =
\frac{\mu}{\mu - \lambda} Z_k T_k,
}
with $Z_k = Z(S_k)$ the value of the process $Z(\cdot)$ at the
start of the $k$-th cycle.

In general, when $U_k$ may not be zero, a similar geometric
construction leads to the conclusion that the two areas are
\eqn{\nonumber
\frac{1}{2} \lambda T_k
\left(T_k - \frac{\mu}{\lambda} U_k\right) +
\frac{1}{2} \frac{\lambda^2}{\mu - \lambda}
\left(T_k - \frac{\mu}{\lambda} U_k\right)^2 \geq
\frac{1}{2} \frac{\lambda \mu}{\mu - \lambda}
\left (T_k - \frac{\mu}{\lambda} U_k\right )^2,
}
and
\eqn{\nonumber
Z_k \Big(T_k + \frac{\lambda}{\mu - \lambda}
\left(T_k - \frac{\mu}{\lambda} U_k\right)\Big) =
\frac{\mu}{\mu - \lambda} Z_k \left (T_k - U_k\right ),
}
respectively. Unconditioning, i.e. taking expectations with respect to the number
of cycles~$M$ and the random variables $T_k$ and $U_k$, $k = 1, \dots, M$,
we find
\eqn{\nonumber
\E \int_{t = 0}^{S_M} Z(t) {\rm d} t \geq
\frac{1}{2} \frac{\lambda \mu}{\mu - \lambda}
\E \Big ( \sum_{k = 1}^{M} (T_k - \frac{\mu}{\lambda} U_k)^2 \Big) +
\frac{\mu}{\mu - \lambda} \E \Big(\sum_{k = 1}^{M} Z_k
(T_k - U_k)\Big),
}
yielding
\begin{align*}
\frac{\E \int_{t = 0}^{S_M} Z(t) {\rm d} t}{\E S_M}
&= \frac{1}{4} \frac{\lambda \mu}{(\mu - \lambda) \E \tyx}
\frac{\E \left( \sum_{k = 1}^{M}
(T_k - \frac{\mu}{\lambda} U_k)^2\right) }{\E M} +
\frac{\mu}{2 (\mu - \lambda) \E \tyx}
\frac{\E \Big( \sum_{k = 1}^{M} Z_k
(T_k - U_k)\Big )}{\E M} \\
&= \frac{1}{4} \frac{\lambda \mu}{(\mu - \lambda) \E \tyx}
\E \Big(\tyx - \frac{\mu}{\lambda} U_\ev \Big)^2 +
\frac{\mu}{2 (\mu - \lambda) \E \tyx}
\E Z (\tyx - U_\ev),
\end{align*}
where the three random variables $\tyx$, $\uy$ and $Z$ have the joint
distribution of the duration of an even renewal period, the amount of
time that the activity process does not reside in the even dominant
state during that period, and the value of the process $Z(\cdot)$ at
the start of that period.

Applying Jensen's inequality and noting that $\tyx$ and $\uy$ are
independent of~$Z$, we obtain
\begin{align*}
\frac{\E \int_{t = 0}^{T_M} Z(t) {\rm d} t}{\E T_M}
&\geq \frac{1}{4} \frac{\lambda \mu}{(\mu - \lambda) \E \tyx}
\left(\E \tyx - \frac{\mu}{\lambda} \E \uy\right)^2 +
\frac{\mu}{2 (\mu - \lambda)\E \tyx}
\E Z \left (\E \tyx - \E \uy\right ) \\
&\geq
\frac{1}{4} \frac{\lambda \mu}{\mu - \lambda} \E \tyx
\Big(1 - \frac{\mu}{\lambda} \frac{\E \uy}{\E \tyx}\Big)^2 +
\frac{\mu}{2 (\mu - \lambda)}
\E Z \Big (1 - \frac{\E \uy}{\E \tyx}\Big).
\end{align*}
In view of~\eqref{eq:Udelta}, we find that for any $\delta > 0$,
the liminf of the latter expression (for large values of~$\s$) is bounded
from below by
\eqn{\nonumber
\frac{1}{4} \frac{\lambda \mu}{\mu - \lambda} \E \tyx
\left(1 - \delta \frac{\mu}{\lambda} \right)^2 +
\frac{\mu}{2 (\mu - \lambda)}
\E Z \left (1 - \delta\right ).
}
Since $Z$ is non-negative and $\delta > 0$ is arbitrary, we deduce
\eqn{\nonumber
\liminf_{\s \to \infty} \frac{\E L}{\E \tyx} \geq
\frac{\lambda}{4} \frac{\mu}{\mu - \lambda},
}
where we also used the fact that the expectation signs in~\eqref{eq:ineq}
can be moved in front of the integrals and then dropped in view
of~\eqref{eq:lave} and~\eqref{eq:zave}.
The statement of the theorem then follows by applying Little's law.

\section{Proof of Theorem~\ref{thm3} (communication height)}
\label{sec:ch}
In this section we obtain the communication height $\G$. In Subsection~\ref{cutset} we introduce cutsets and contours, which we use in Subsection~\ref{clustricro} to distinguish among feasible configurations. We then explore the geometric properties of the set $\SS$ in Subsection~\ref{setS} and we present the proof of Theorem~\ref{thm3} in Subsection~\ref{proofthm3}.

\subsection{Cutsets and contours}
\label{cutset}
We start by introducing some notation. Given $S \subseteq V$, define $S^c:=V\setminus S$, the set $\pam S$ of vertices in $S$ which are adjacent to a vertex in $S^c$, the set $\pap S$ of vertices in $S^c$ which are adjacent to a vertex in $S$ (in particular for $v\in V$, we write $\partial v$ for $\pap \{v\}$) and the sets $S^+:=S \cup \pap S$, $S^\ev:=S \cap \ev$ and $S^\ov:=S \cap \ov$. Let $\nabla S:=\{(v,w) \in E ~|~ v \in \pam S, \, w \in \pap S\}$ denote the set of edges having one end in $S$ and the other in $S^c$.

An independent set $I \in \Omega$ consists of vertices in $I^\ev$ and/or vertices in $I^\ov$. In between these regions there is always an unoccupied two-layer interface, which corresponds to a collection of edge cutsets. Recall that $\g \subseteq E$ is an \textit{edge cutset} if the graph $(V,E\setminus \g)$ is disconnected and it is said to be \textit{minimal} if it is inclusion-wise minimal. We describe a way of associating with each independent set $I \in \Omega$ a collection of minimal edge cutsets, following the approach and some ideas of~\cite{B99, G08, R06}.

For a given $I\in \Omega$, the vertices in $V$ can be partitioned in two sets, $\RO(I)=I^\ov \cup (\ev \setminus I^\ev)$ and its complement $\RE(I)=I^\ev \cup (\ov \setminus I^\ov)$. The collection $\RO(I)$ ($\RE(I)$) consists of one or more connected components to which we will refer as \textit{odd regions} (\textit{even regions}). From this point onward our focus will be on the odd regions only and, unless stated otherwise, all regions that will be mentioned will be odd. 

For each region $R \subseteq \RO(I)$, we consider the corresponding edge cutset $\g_R:=\nabla(R)$, which can consist of several disjoint components. A well-known identity relates the size of a specific edge cutset to the number of even and odd vertices present in its interior, see~\cite{BGTR12,B99,G08}. This can be generalized and, in the case of a region $R\subseteq \RO(I)$, it reads
\eqn{\label{eq:ci}
|\g_R|=|\nabla(R)|=4(|R^\ev|-|R^\ov|).
}
One way to interpret~\eqref{eq:ci} is the following: For a given region $R \subseteq \RO(I)$, $|\g_R|$ is equal to $4$ times the difference between the number of even inactive nodes in $R$, i.e. $|R^\ev|$, and the number of odd active nodes in $R$, i.e. $|R^\ov|$. With this interpretation in mind, the idea behind the proof is easily explained. Indeed one way to count all the edges of $\nabla(R)$, i.e. the ``outgoing'' edges of $R$, is counting first how many edges the subgraph induced by $R^+$ has and then subtracting the number of its ``internal'' edges. The total number of edges of $R^+$ is $4|R^\ev|$, while the number of internal ones is $4|R^\ov|$, since each of them has precisely one endpoint in $R^\ov$. 

Consider the dual graph $G'$ of the graph $G$, which is a discrete torus of the same size. For every region $R \subseteq \RO(I)$, we associate with the edge cutset $\g_R$ the edge set $c_R$ on $G'$ which consists of all the edges of $G'$ orthogonal to edges in $\g_R$. The set $c_R$, to which we will refer as the \textit{contour} of $R$, is a collection of one or more disjoint piecewise linear closed curves. By construction $|c_R|=|\g_R|$ and $|c_R|$ has the natural interpretation as the total length of the contour of $R$. Define $\LL(I):=\bigcup_{R \subseteq \RO(I)} c_R$ to be the collection of contours of all the regions in $\RO(I)$ and $l(I)$ to be their total length, i.e. $l(I):=\sum_{R \subseteq \RO(I)} |c_R|$. The total contour length of a configuration $I\in \Omega$ is related to its efficiency gap $\D(I)$ as follows:
\eqn{\label{eq:contour}
l(I)= 4\cdot \D(I),
}
which follows from~\eqref{eq:ci} and the definition of $\RO(I)$. Indeed
\eqan{
l(I)
&=\sum_{R \subseteq \RO(I)} |c_R| =\sum_{R \subseteq \RO(I)} |\g_R| =\sum_{R \subseteq \RO(I)} 4(|R^\ev |-|R^\ov|) =4 \,\Big ( \sum_{R \subseteq \RO(I)} |R^\ev |-\sum_{R \subseteq \RO(I)} |R^\ov | \Big )\nonumber \\
&=4 \left ( |\RO(I)^\ev |-|\RO(I)^\ov | \right )=4 \left (|\ev\setminus I^\ev|-| I^\ov |\right )=4 \left (|\ev|-|I|\right )=4 \cdot \D(I). \nonumber
}

To help the reader, we will present later some examples of independent sets of the graph $G$, with the corresponding regions and contours. Displaying a configuration $I\in\Omega$, a vertex is drawn as a disk, if it belongs to $I$, and otherwise as a circle. Moreover we use the following color conventions:
\begin{center}
{\renewcommand{\arraystretch}{1}
  \begin{tabular}{| c | c | c | }
    \hline
    vertex		& odd		& even 	      \\ \hline
    active 		& $\evena$ 	& $\odda$  \\ \hline
    inactive 	& $\eveni$	& $\oddi$ \\
    \hline
  \end{tabular}
}
\end{center}
With these conventions, all regions in $\RO(I)$ end up being colored in black (and those in $\RE(I)$ in gray).

\subsection{Clusters, stripes and crosses}
\label{clustricro}
In this subsection we present a partition of the state space $\Omega$ exploiting geometric properties of the configurations. Let $\Lambda_\ov$ be the graph whose vertex set is $\ov$ and such that two odd vertices are connected if and only if they have distance $2$ in $G$ and define $\Lambda_\ev$ analogously. A closed path or curve is said to be \textit{contractible} in a graph $H$ if it can be continuously deformed or contracted in $H$ to a single vertex. Given an independent set $I \in \Omega$, a region $R\subseteq \RO(I)$ is called
a \textit{stripe} if there exists a curve in $c_R$ which is non-contractible in $G'$,
a \textit{cross} if all the curves in $c_R$ are contractible in $G'$ and there is a non-contractible closed path in $\Lambda_\ov \cap R^\ov$, and a
\textit{cluster} if all the curves in $c_R$ are contractible in $G'$ and all the closed paths are contractible in $\Lambda_\ov \cap R^\ov$. Notice that if $R$ is a stripe, by parity there must be two non-contractible closed curves in $c_R$. If instead $R$ is a cross, then one can show that there exist two non-contractible paths in $\Lambda_\ov \cap R^\ov$ which are not homotopic and intersect each other (this motivates the name \textit{cross}). Figure~\ref{fig:regions} below shows some examples of these types of regions.

\begin{figure}[!ht]
\centering
\subfigure[Cluster]{
\includegraphics[angle=90,scale=0.19]{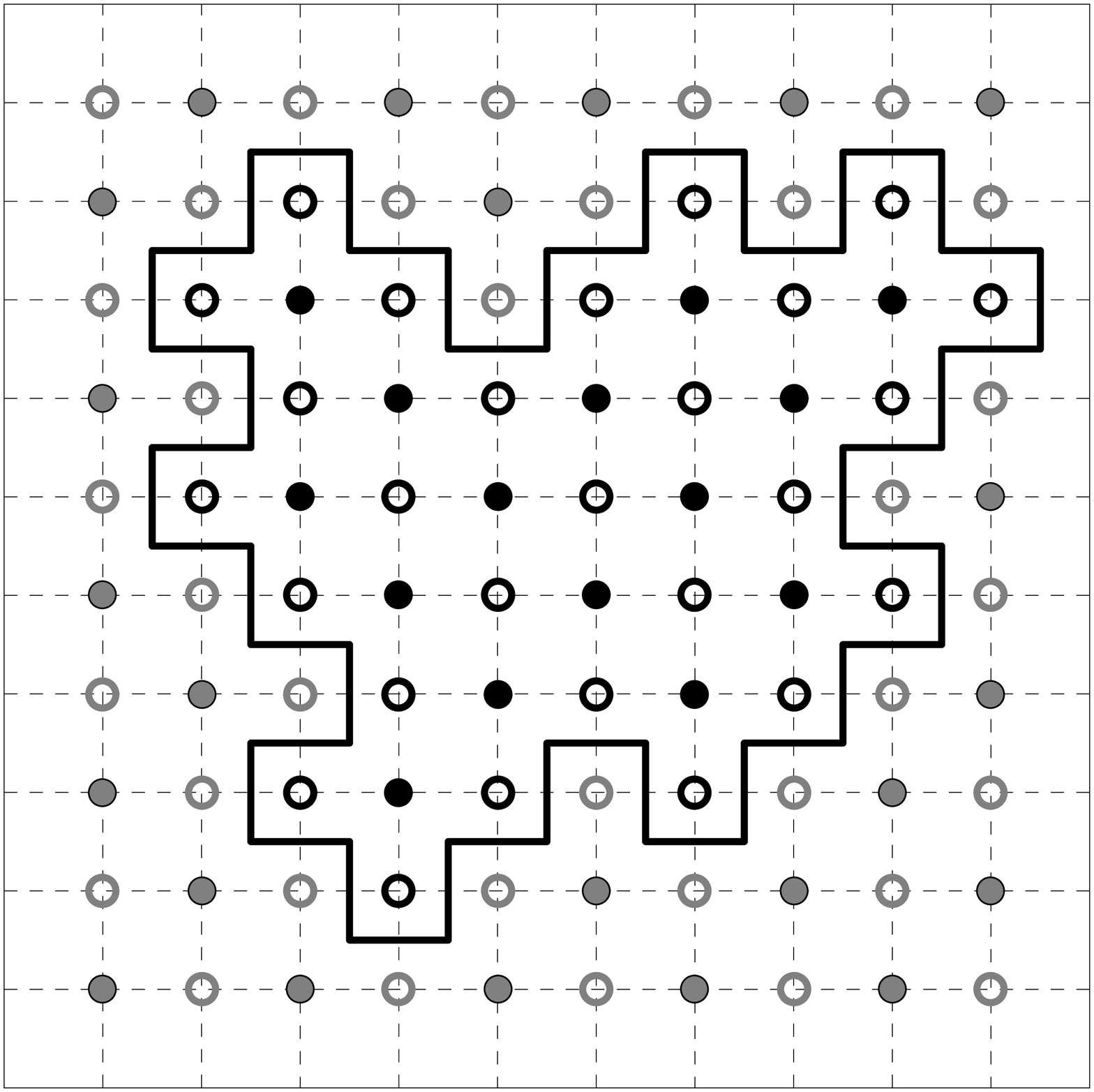}
}
\hspace{-0.9cm}
\subfigure[Cluster with holes]{
\includegraphics[angle=90,scale=0.19]{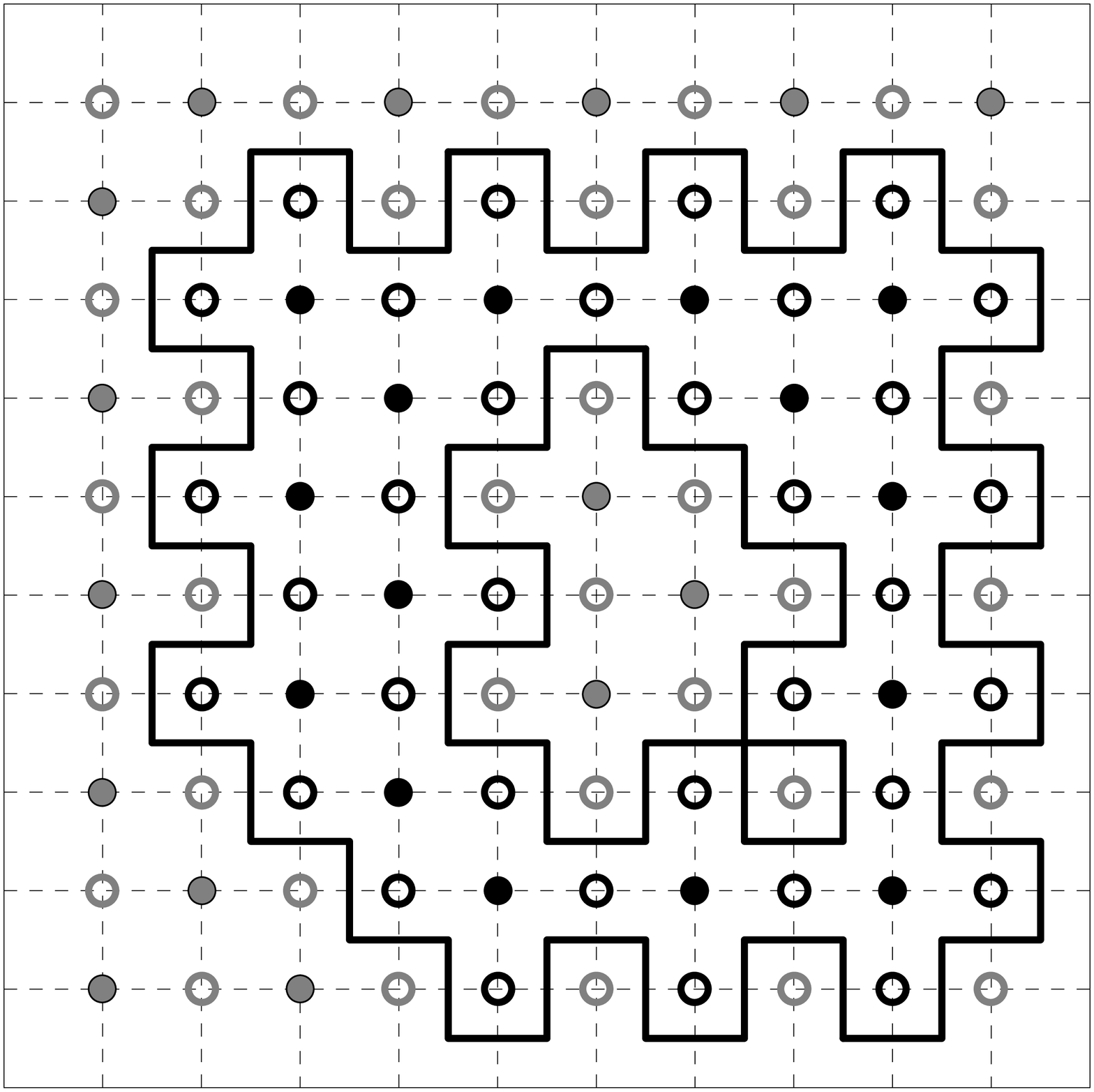}
}
\hspace{-0.9cm}
\subfigure[Cross $\ov$]{
\includegraphics[scale=0.19]{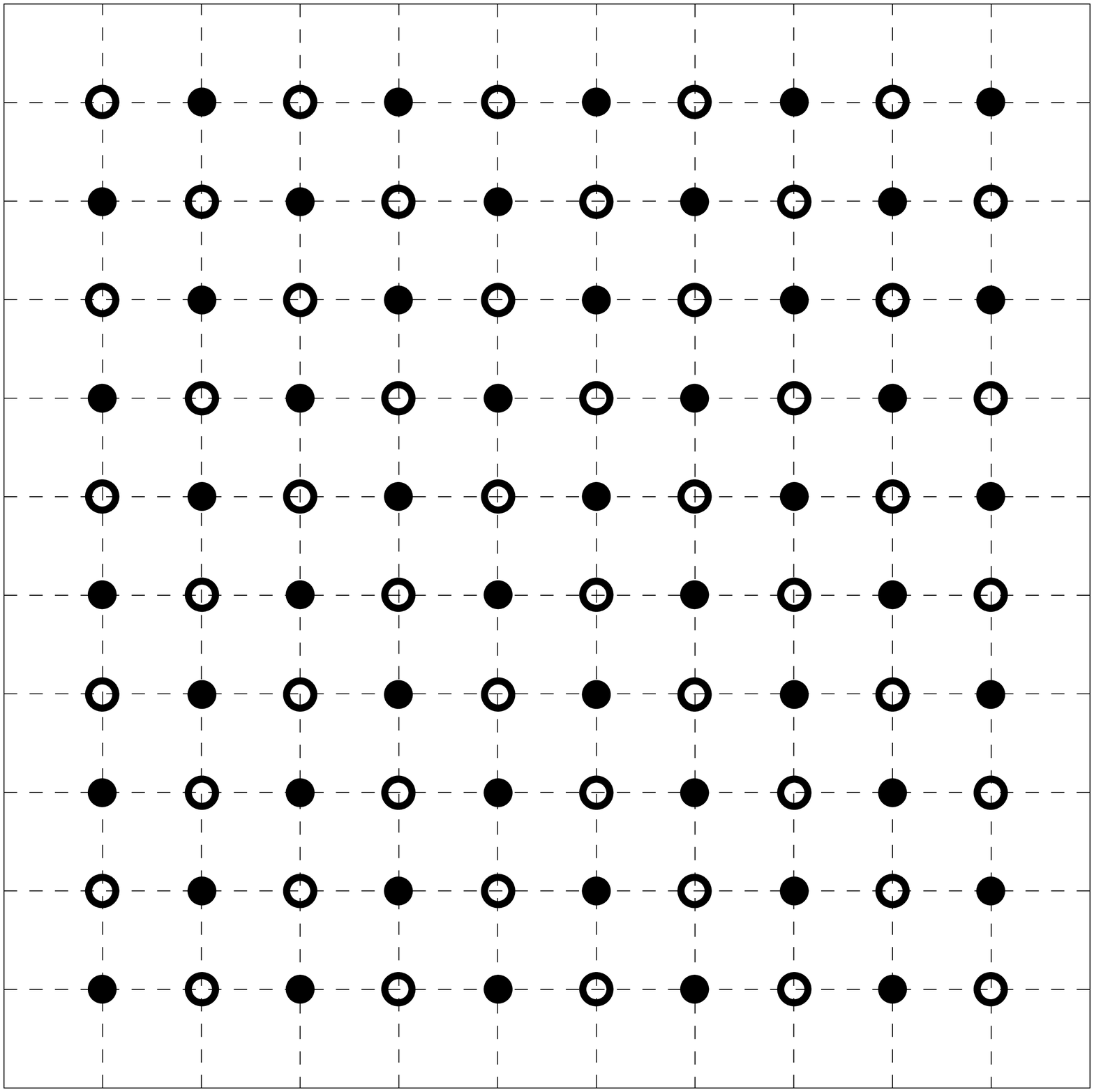}
}
\\
\centering
\subfigure[Cross]{
\includegraphics[angle=90,scale=0.19]{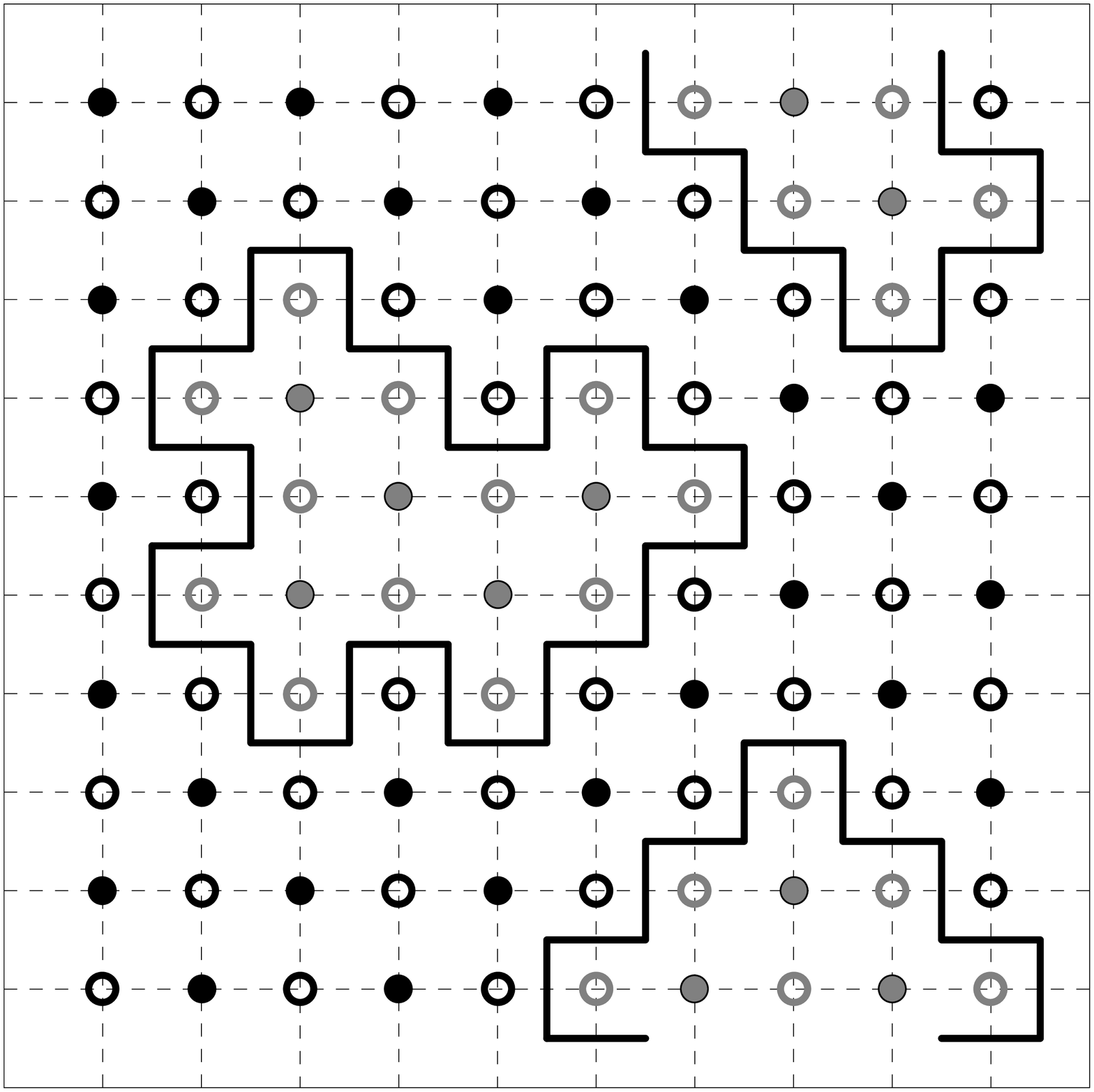}
}
\hspace{-0.9cm}
\subfigure[Stripe]{
\includegraphics[angle=90,scale=0.19]{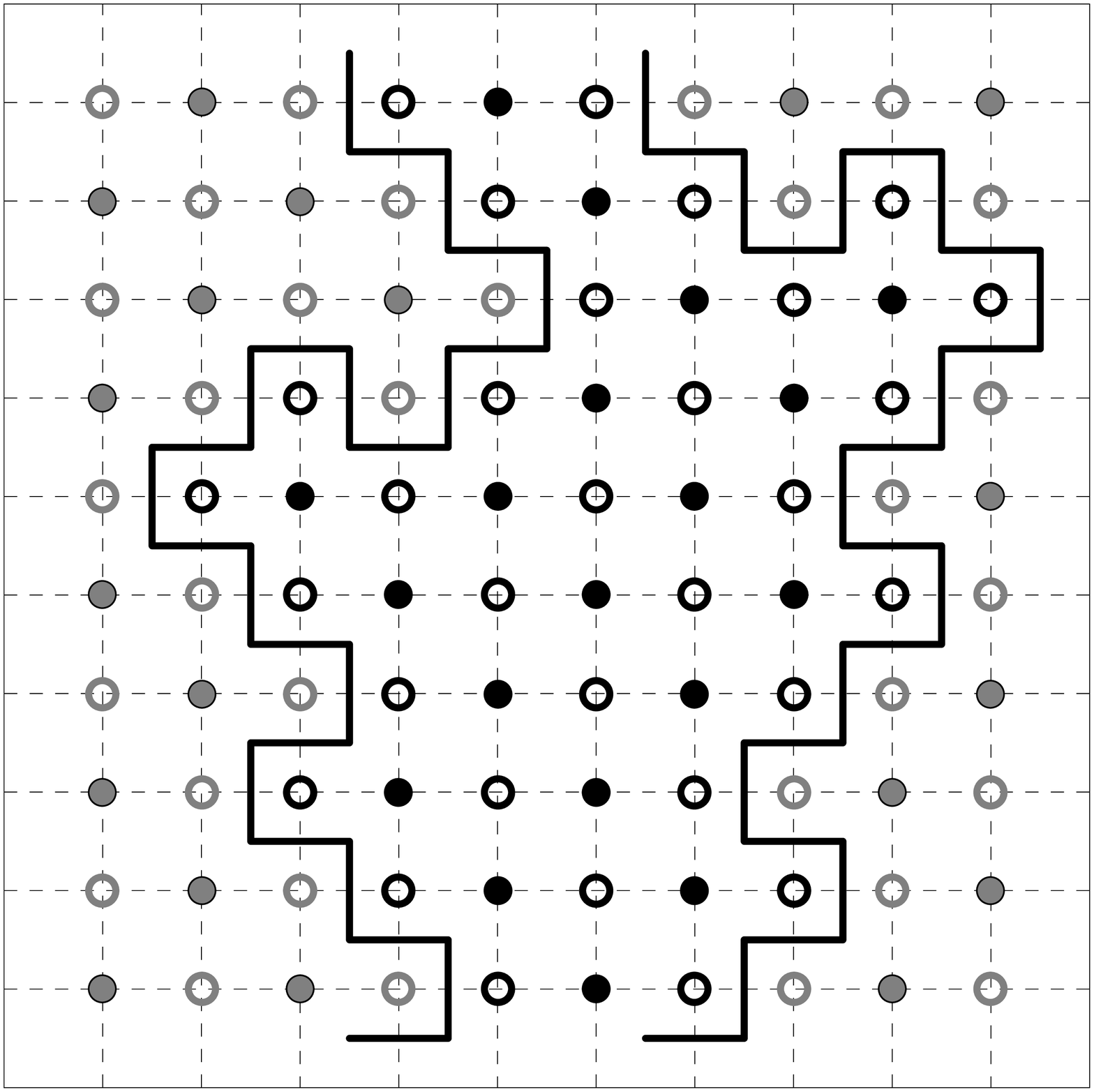}
}
\hspace{-0.9cm}
\subfigure[Stripe with holes]{
\includegraphics[angle=90,scale=0.19]{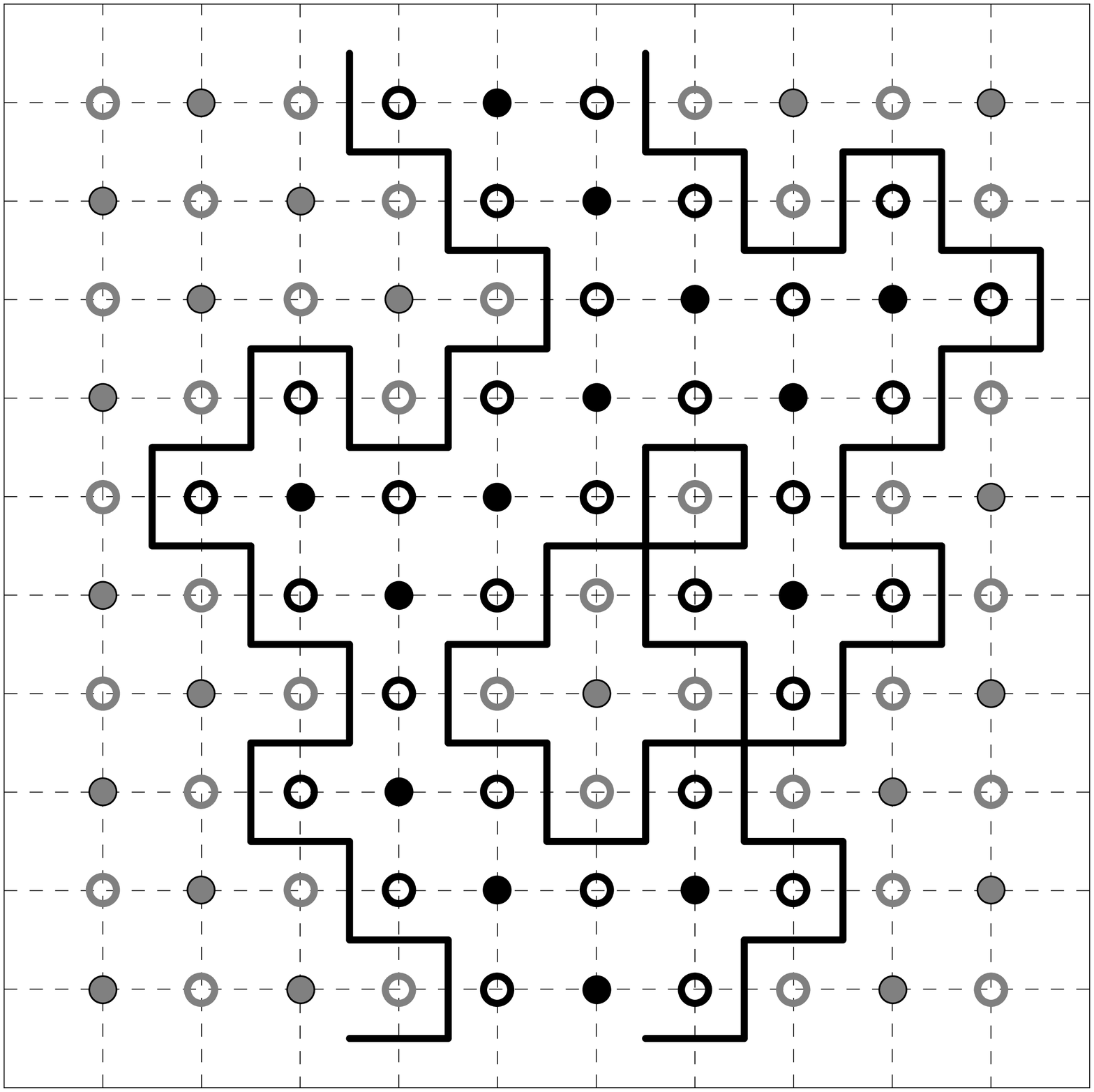}
}\\
\caption{Examples of configurations in $\Omega$ for $L=10$, with the respective contours.}
\label{fig:regions}
\end{figure}

We partition the state space $\Omega$ into the following three subsets: $\Omega_{s}$ the collection of independent sets which have at least one stripe; $\Omega_{cr}$ the collection of independent sets which have a cross; $\Omega_{cl}$ the collection of independent sets whose regions, if any, are clusters only. The subset $\Omega_{cl}$ is disjoint from the other two by definition. Then, to guarantee that $\Omega_{s} \cup \Omega_{cr} \cup \Omega_{cl}$ is indeed a partition of $\Omega$, we need to prove that there are no configurations $I\in\Omega$ which display simultaneously a cross and a stripe, i.e. $\Omega_{s} \cap \Omega_{cr}= \emptyset$. This is indeed the case, since the non-contractible curve, which the contour of a stripe displays, is surrounded by an unoccupied two-layer interface that precludes the existence of one the two non-contractible paths on $\Lambda_\ev$ which a cross must have. 

Denote by $A \triangle B$ the symmetric difference between two sets $A$ and $B$. Consider the collection of independent sets which exhibit a \textit{critical cross}, defined as $\Omega_{cc}:=\{ I \in \Omega_{cr} ~|~ \exists \, I' \in \Omega_{cl} : |I \triangle I'|=1 \}$. In words, a critical cross is a cross which can be ``broken'' deactivating just a single node directly into clusters, without becoming a stripe as intermediate stage. 

The next lemma gives some lower bounds on the length of the contour of certain type of regions, and will be crucial to understand the geometrical properties of the set $\SS$.

\begin{lemma}[Minimal contour length for stripes and critical crosses]\label{lem:contourstripecc}
Let $R \subseteq \RO(I)$ be a region, for some $I\in\Omega$. If $R$ is a stripe, then $|c_R| \geq 4 \gs$. If $R$ is a critical cross, then $|c_R|\geq 8 \gs -12$.
\end{lemma}

This lemma implies that if $I$ has a stripe, then $\D(I) \geq \gs$, while if $I$ has a critical cross, then $\D(I) \geq 2\gs-3$.

\begin{proof}
{\rm a)} Consider a stripe $R$. By definition, its contour $c_R$ has at least one non-contractible curve, but by parity, there are two disjoint non-contractible curves. A non-contractible curve has length at least $2 \gs$, because it should be long enough in one direction to wind around the $L\times L$ toric grid. Moreover, every curve of the contour has by construction an equal number of vertical and horizontal edges, which means that both curves have length at least $2 \gs$. Hence $ |c_R| \geq 4 \gs$ and the proof is concluded.

{\rm b)} Let $I\in\Omega_{cc}$ be an independent set which exhibits a critical cross $R$ and consider a configuration $I' \in \Omega_{cl}$ such that $|I \triangle I'|=1$, i.e. $I' \in \pap \Omega_{cc}$. A single cluster cannot evolve in a single step into a cross, so the configuration $I'$ must have at least two clusters, say $R_1$ and $R_2$, from which the critical cross $R$ will be created. Moreover such clusters must have specific features: $R_1$ and $R_2$ must miss the same odd vertex $v \notin I'$ to become (individually) a stripe, see Figure~\ref{fig:criticalcrossandneigh} below. This means that $R_1$ and $R_2$ are clusters ``stretched in different directions''. More precisely, one of them, say $R_1$, has a projection of length $\gs-1$ in the horizontal direction and the other one, $R_2$, a projection of length $\gs-1$ in the vertical direction. This fact implies that both $R_1$ and $R_2$ have a contour of length at least $4(\gs-1)$. Indeed to have a close curve, at each latitude and longitude there must be two edges, and a contour has the same number of vertical and horizontal edges. Thus $l(I') \geq |c_{R_1}|+|c_{R_2}|\geq 8(\gs-1)$ and, by construction, $l(I)=|c_{R_1}|+|c_{R_2}|-4 \geq 8(\gs-1)-4 = 8\gs -12$.
\end{proof}

\begin{figure}[!ht]
\centering
\subfigure[Configuration $I\in \Omega_{cc}$]{
\includegraphics[angle=90,scale=0.19]{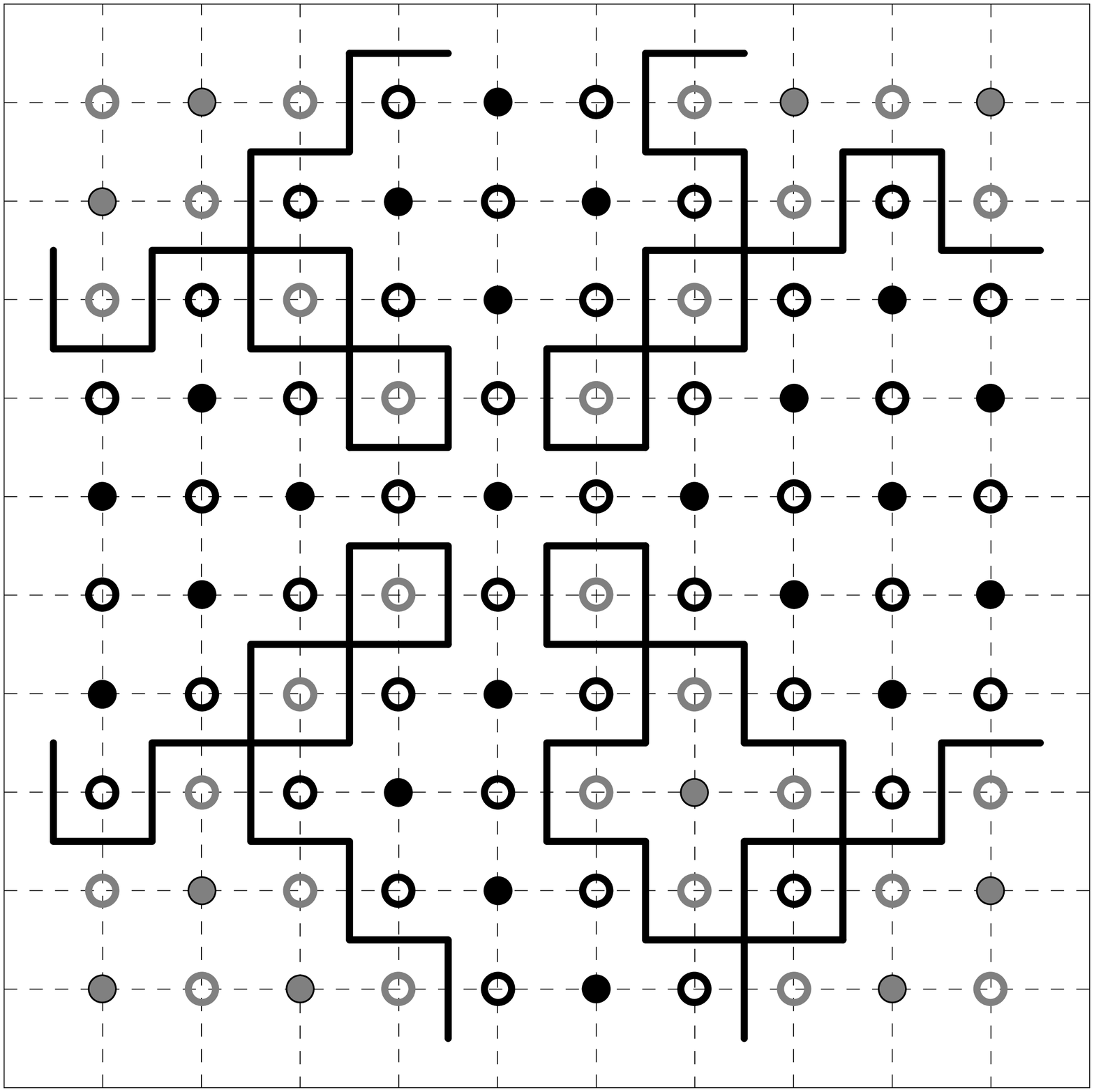}
}
\,
\subfigure[Configuration $I' \in \pap \Omega_{cc}$, $|I \triangle I'|=1$]{
\includegraphics[angle=90,scale=0.19]{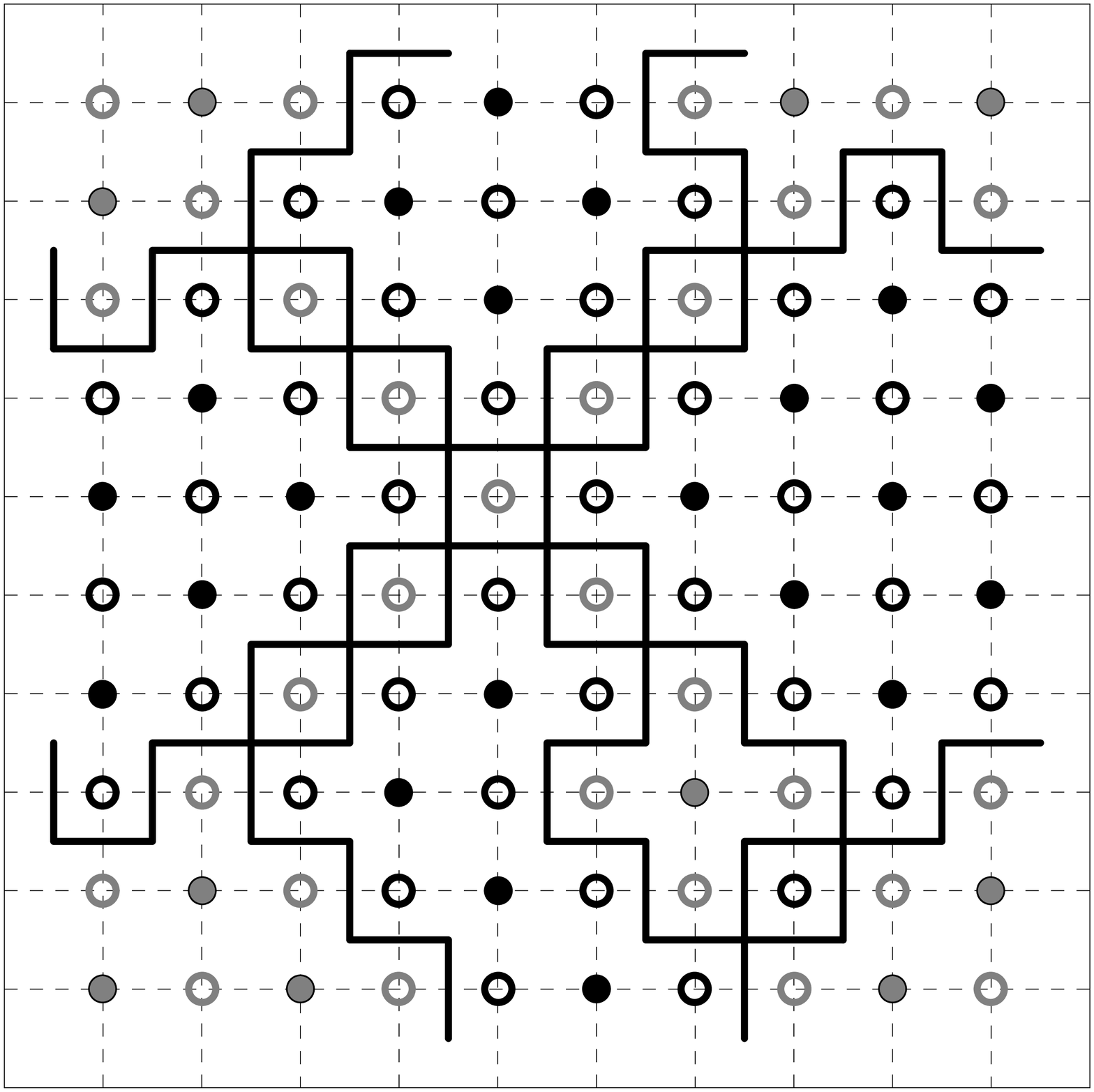}
}
\label{fig:criticalcrossandneigh}
\end{figure}

\subsection{The set \texorpdfstring{$\SS$}~ and its properties}
\label{setS}
Consider two configurations $I,I' \in \Omega$, $I \neq I'$. By construction, if the Markov process $\{X(t)\}_{t\geq 0}$ can make the transition $I \to I'$ in a single step, then $ |I \triangle I'|=1$ and hence $\D(I')=\D(I)\pm 1$. In particular $\D(I')\neq\D(I)$.

The {\it bottom} $\cF(S)$ of a non-empty set $S \subseteq\Omega$ is the set of the largest independent sets in $S$, i.e. $\cF(S) = \{ I \in S : \D(I)=\min_{I' \in S} \D(I') \}$. The independent sets in $\cF(S)$ have all the same efficiency gap, which we will denote by $\D(\cF(S))$. Thanks to~\eqref{eq:contour} and the definition of the communication height, for any integer $c$
\eqn{\label{eq:chcl}
\phi(I,I') \leq c \Leftrightarrow  \D(J) \leq 4 \cdot c, \qquad \forall \, J \in \o: I \to I'.
}

A {\it non-trivial cycle} is a connected set $C$ such that $\max_{I \in C} \D(I) < \D(\cF(\pap C))$. Any singleton that is not a non-trivial cycle is called {\it trivial cycle}. The {\it depth} of a cycle $C$ is given by $D(C):=\D(\cF(\pap C))-\D(\cF(C))$.

Recall the definition~\eqref{eq:defS} of the set $\SS=\{ I \in \Omega ~|~ \phi(\ev,I) \leq \gs \}$. Thanks to~\eqref{eq:chcl}, the condition $ \phi(\ev,I) \leq \gs$ is equivalent to $l(I') \leq 4 \cdot \gs$ for all configurations $I' \in \o: \ev \to I$. Moreover, $\SS \neq \Omega$, because for instance an independent set $J$ consisting of a single vertex does not belong to $\SS$. Indeed $\D(J)=\gs^2/2-1 > \gs$ for $\gs > 2$. This fact and the irreducibility of the process on $\Omega$ imply that $\pap \SS \neq \emptyset$.

\begin{lemma}[Stripes and crosses do not belong to $\SS$]\label{lem:nosnocc} Consider an independent set $I \in \Omega$.
\begin{itemize}
\item[{\rm a)}] If $I$ has at least one stripe, then any path $\o: I \to \ev$ satisfies $\max_{I' \in \o} \D(I') > \gs$.
\item[{\rm b)}] If $I$ has a cross, then any path $\o: I \to \ev$ satisfies $\max_{I' \in \o} \D(I') > \gs$.
\end{itemize}
\end{lemma}

\begin{proof}
{\rm a)} Consider a path $\o: I \to \ev$, $\o=(I_1,\dots,I_n)$, $n\in\N$, $I_1=I$, $I_n=\ev$. There exists $1 \leq m \leq n-1$ such that $I_m$ is the first configuration belonging to $\Omega_{cl}$ along the path $\o$. In the configuration $I_{m-1}$ there is either a stripe $S$ or a critical cross $C$.

Suppose that $I_{m-1}$ has a stripe $S$. Its contour satisfies $|c_S|\geq 4 \gs$, as proved above in Lemma~\ref{lem:contourstripecc}. Due to our assumptions, the stripe $S$ must ``break'' in a single step into clusters. Recall that in a single transition, we can either add or remove a vertex to the current independent set. The action of adding an (even or odd) vertex cannot break the stripe $S$. Therefore, the only way to break a stripe $S$ is by removing at least one vertex (on its inside), obtaining a new independent set $I'$ which satisfies $\D(I')=\D(I)+1\geq \gs+1$.

Suppose instead that $I_{m-1}$ has a cross $C$. By construction, $C$ must be a critical cross and then Lemma~\ref{lem:contourstripecc} implies that $\D(I_{m-1}) \geq 2\gs -3$, which is greater than $\gs$, since $\gs \geq 4$.

{\rm b)} Consider a path $\o': I \to \ev$, $\o'=(I_1,\dots,I_{n'})$, $n'\in\N$, $I_1=I$, $I_{n'}=\ev$. There exists $1 \leq m' \leq n'-1$ such that $I_{m'}$ is the first configuration belonging to $\Omega_{cl}$ along the path $\o'$. In the configuration $I_{m'-1}$ there is either a stripe $S$ or a critical cross $C$ and then we can argue as in ${\rm a)}$.
\end{proof}

In particular, this latter lemma implies that $\SS \subseteq \Omega_{cl}$. We remark that $\SS \neq \Omega_{cl}$. Indeed there exist independent sets $I \in \Omega_{cl}$ which have a cluster with a contour larger than $4 \gs$ and therefore do not belong to $\SS$. The next proposition gives two crucial properties of the set $\SS$, following the general method proposed in~\cite[p.~614]{MNOS04}.

\begin{proposition} \label{prop:S}
The set $\SS$ satisfies the following two properties:
\begin{itemize}
\item[{\rm a)}] $\SS$ is a connected set, $\ev \in \SS$ and $\ov \notin \SS$;
\item[{\rm b)}] There exists a path $\o^*: \ev \to \ov$ such that the maximum of $\D(\cdot)$ is reached in $\cF(\pap \SS )$, namely
\eqn{\label{eq:propb} \arg \max_{\o^*} \D \cap \cF (\pap \SS) \neq \emptyset.}
\end{itemize}
\end{proposition}

\begin{proof}
{\rm a)} Clearly $\ev \in \SS$ and $\SS$ is connected by construction. We claim that every path $\o: \ev \to \ov$ satisfies $\max_{I \in \omega} \D(I) > \gs$, which implies that $\ov \notin \SS$. Indeed $\ev \in \Omega_{cl}$, $\ov \in \Omega_{cr}\setminus \Omega_{cc}$ and therefore along any path $\o: \ev \to \ov$ there must be a configuration showing a stripe or a critical cross and Lemma~\ref{lem:nosnocc} proves our claim.

{\rm b)} The reference path $\o^*$ can be constructed as follows: Starting from $\ev$, we gradually create a linear cluster, until it contains $\gs/2-1$ odd active nodes. This configuration, say $I$, still belongs to $\SS$, since its only cluster has a contour of length $4\gs-4$ and hence $\D(I)=\gs-1$. In order to add the odd vertex $v$ which is missing to create a stripe, we first need to remove sequentially the two even vertices $v_1,v_2$ in  $I \cap \partial v$. Let us call $I_1$ and $I_2$ the configurations corresponding to these two intermediate steps and $I_3$ the configuration where the stripe finally appears, adding the vertex $v$.
By construction, $I_1 \in \SS$, with $\D(I)=\gs$, while $\D(I_2)= \gs+1$ and hence $I_2 \in \pap \SS$. From $I_3$ it is possible to reach the configuration $\ov$ visiting only configurations $J\in\Omega$ such that $\D(J)\leq \gs+1$, just gradually expanding the stripe column by column. Therefore $I_2 \in \arg\max_{\o^*} \D \cap \cF(\pap \SS)$.
\end{proof}

\subsection{Proof of Theorem~\ref{thm3}}
\label{proofthm3}
Proposition~\ref{prop:S}(a) implies that the process should exit from the set $\SS$ to reach the state $\ov$ and so it must visit a configuration $I \in \pap \SS$. Therefore
$\phi(\ev,\ov) \geq \D(\cF(\pap \SS))$.
Moreover the existence of a path $\o^*: \ev \to \ov$ such that the maximum of $\D(\cdot)$ is reached in $\cF(\pap \SS )$, guaranteed by Proposition~\ref{prop:S}(b), gives
$\phi(\ev,\ov) \leq \max_{I \in \o^* } \D(I)$,
and by~\eqref{eq:propb} the two bounds coincide and are equal to $\D(\cF(\pap \SS))$. The definition of $\SS$ yields that any $I \in \cF(\pap \SS)$ is such that $\D(I)=\gs+1$ and the conclusion follows.

\section{Proof of Theorem~\ref{thm1} (Transition time between dominant states)}
\label{sec:thm31}
In this section we prove Theorem~\ref{thm1}, exploiting the communication height $\G$, together with a well-known result about exit times from cycles, namely Theorem 6.23 in~\cite{OV05}. In order to do so, we will consider the uniformized discrete-time version of the activity process. In more detail, we construct a discrete-time Markov chain $\{\tilde{X}(t) \}_{t\in \N}$ starting from the continuous-time Markov process $\{X(t)\}_{t \geq 0}$ by means of uniformization at rate $q_{{\rm max}}=\gs^2 \xi$, where $\xi=\max \{p \mu,\nu\}$. The transition probabilities of the new Markov chain are as follows:
\[
p(I,J) =
\left\{\begin{array}{ll}
\nu/q_{{\rm max}},  & \text{ if } |I \triangle J|=1, I \subset J,\\ 
p \mu /q_{{\rm max}}, & \text{ if } |I \triangle J|=1, J \subset I,\\  
1 - \sum_{I'\neq I} p(I,I'), & \text{ if } I=J, \\
0, & \text{ otherwise.}
\end{array}\right.
\]
Since we are interested in the regime $\s \to \infty$, we will assume that $\s \geq 1$, so that $\xi=\nu$ and the above transition probabilities may be written in the form
\[
p(I,J) =
\begin{cases}
c(I,J) \s^{-[H(J)-H(I)]^+}, & \text{ if } I\neq J,\\
1 - \sum_{I' \neq I} p(I,I'), & \text{ if } I=J,
\end{cases}
\]
where $H: \Omega \to \mathbb R$ is the \textit{Hamiltonian} defined as $H(I):=-|I|$ and $c$ is the connectivity function $c: \Omega^2 \setminus\{(I,I) : I \in \Omega \} \to [0,1]$ defined as
\[
c(I,J) =
\begin{cases}
1/\gs^2, & \text{ if } |I \triangle J|=1,\\
0, & \text{ otherwise.}
\end{cases}
\]
The process $\{\tilde{X}(t) \}_{t\in \N}$ is a reversible Markov chain which, taking $\s = e^\b$, satisfies condition M in~\cite[p.~336]{OV05} and thus fits in that framework. Moreover, its stationary distribution is the same as that of the original continuous-time process $\{X(t)\}_{t \geq 0}$, see~\eqref{eq:pih}, and can be rewritten as
\[
\pi(I)=\frac{e^{-\b H(I)}}{\sum_{J \in \Omega} e^{-\b H(J)}}, \quad I \in \Omega.
\]
Let $(Y_n)_{n \in \N}$ be a sequence of i.i.d.~exponential random variables with mean $1/q_{{\rm max}}$. If $\tilde{X}(0)=X(0)$ and $n(t):= \sup \{ m : \sum_{n=1}^m Y_n \leq t\}$, then $ \{X(t)\}_{t\geq 0}\ed\{\tilde{X}(n(t))\}_{t\geq 0}$.

For $I \in \Omega$ and $A \subseteq \Omega$, let $\t^{I}_{A}(\s):=\inf \{t> 0 : \tilde{X}(t) \in A\}$ be the first hitting time of the set $A$ for the process $\tilde{X}(t)$ starting in $I$ at $t=0$ and let $T_{I\to A}(\s):=\inf \{t> 0 : X(t) \in A\}$ be its continuous-time counterpart, namely the first hitting time of the set $S$ for the process $X(t)$ starting in $I$ at $t=0$. These hitting times are closely related:
\[
T_{I \to A} \ed \sum_{n=1}^{\t^I_A} Y_n.
\]
In particular,
\eqn{\label{eq:Ectdt}
\E T_{I \to A} (\s) = \frac{1}{q_{{\rm max}}} \E \t^I_A (\s).
}

Let us now focus on the Markov chain $\{\tilde{X}(t) \}_{t\in \N}$. As shown in Proposition~\ref{prop:S}, this process must exit from the set $\SS$ in order to make the transition $\ev \to \ov$. In particular it must hit the set $\pap \SS$, the external boundary of $\SS$, before reaching $\ev$. Therefore $ \t_{\ov}^{\ev} \stl \t^{\ev}_{\partial \SS}$ and hence
\eqn{\label{eq:tailineq}
\pr{ \t_{\ov}^{\ev} \geq e^{\b (\G-\e)} }  \geq \pr{ \t^{\ev}_{\pap \SS} \geq e^{\b (\G-\e)} }.
}
Theorem 6.23 in~\cite{OV05} says that for any $\e>0$, there exists $k(\e)>0$ and $\b_0 >0$, such that for all $\b \geq \b_0$ and for any $I \in \SS$,
\eqn{\label{eq:exitfromS}
\pr{\t^{I}_{\pap \SS} < e^{\b (D(\SS) - \e)}} \leq e^{-k(\e) \b},
}
where $D(\SS)$ is the depth of the cycle $\SS$, which is equal to $\G$ by construction. Equations~\eqref{eq:tailineq} and~\eqref{eq:exitfromS} imply that for every $\e >0$ there exists $k(\e)>0$ and $\b_0 >0$, such that for all $\b \geq \b_0$
\eqn{\nonumber
\pr{ \t_{\ov}^{\ev} \geq e^{\b (\G-\e)} } \geq 1-e^{- \b k(\e)}.
}
In particular, $\exists \, \b_1 > \b_0$ such that $ \pr{ \t_{\ov}^{\ev} \geq e^{\b (\G-\e)} } \geq 1/2$. Thus
\eqn{\nonumber
\E \t_{\ov}^{\ev} \geq \E \left( \t_{\ov}^{\ev} ~|~ \t_{\ov}^{\ev} \geq e^{\b (\G-\e)} \right ) \pr{ \t_{\ov}^{\ev} \geq e^{\b (\G-\e)} } \geq \frac{1}{2} e^{\b (\G-\e)}.
}
Therefore, for $\b > \b_1$, 
\[\log \E \t_{\ov}^{\ev} \geq \b (\G-\e) \log 2^{-1}=  \b (\G-\e) - \log 2,\]
and hence
\eqn{\nonumber
\frac{1}{\b} \log \E \t_{\ov}^{\ev} \geq \G-\e - \frac{\log 2}{\b},
}
or equivalently, by replacing $\b=\log \s$, for $\s$ sufficiently large,
\eqn{\nonumber
\frac{\log \E \t_{\ov}^{\ev}}{\log \s} \geq \G-\e - \frac{\log 2}{\log \s}.
}
Taking the liminf for $\s \to \infty$ on both sides, we get
\eqn{\nonumber
\liminfnu \frac{ \log \E \t_{\ov}^{\ev} }{\log \s} \geq \G-\e,
}
and, since $\e$ is arbitrary, it follows that
\eqn{\label{eq:Etd}
\liminfnu \frac{\log \E \t_{\ov}^{\ev}}{\log \s} \geq \G.
}
Using the fact that $q_{{\rm max}}=\gs^2 \nu=\gs^2 \s p \mu$, \eqref{eq:Ectdt} implies that 
\[
\log \E \tyx = \log \E \t_{\ov}^{\ev} - \log (\gs^2)- \log \nu,
\]
and hence
\eqan{
\liminfnu \frac{\log \E \tyx}{\log \s} &= \liminfnu \frac{\log \E \t_{\ov}^{\ev} - \log (\gs^2) - \log \s - \log p -\log \mu}{\log \s} \nonumber\\
&=\liminfnu \frac{\log \E \t_{\ov}^{\ev}}{\log \s} - 1 + \alpha \nonumber\\
&\stackrel{\eqref{eq:Etd}}{\geq} \G - 1 + \alpha,\nonumber
}
with $\alpha=\liminfnu \frac{\log p}{\log p - \log \nu}$. The conclusion then follows from Theorem~\ref{thm3}, which states that $\G=\gs+1$.

\section{Concluding remarks}
\label{concl}

We have obtained delay bounds for random-access grid networks, 
which show that the delays grow dramatically with both the load
and the dimension of the network.
For transparency, we have focused on symmetric grid networks,
but the proof techniques and delay bounds are expected to extend
to a far broader range of scenarios.

In the present paper we have assumed the activation rate~$\nu$
and back-off probability~$p$ to be static parameters.
A natural question is whether dynamic activation rates and back-off
probabilities can potentially achieve better delay performance.
Specifically, various algorithms have recently been proposed where
the back-off probability at each node is dynamically adapted over
time as a (decreasing) function $p(q)$ of the instantaneous queue-length~$q$ \cite{JSSW10,RSS09,SS12,SST11}.
Remarkably, for suitable choices of the function $p(\cdot)$ such
algorithms are guaranteed to provide maximum stability in arbitrary
topologies, without any explicit knowledge of the arrival rates.
However, lower bounds in~\cite{BBvL13} demonstrate that for such
choices of $p(\cdot)$ the delays are of the order
$\exp\big(\frac{1}{2 (1 - \rho)}\big)$, growing even faster with the load than the
$\left(1 - \rho \right)^{-L}$ scaling we obtained for
the $L \times L$ toric grid.

In the specific case of an $L \times L$ grid, the fluid-limit
results in~\cite{GBW13} suggest that maximum stability would
actually be maintained as long as the function $p(q)$ decays
no faster than $q^{- 2 / L^2}$.
The lower bounds in~\cite{BBvL13} then indicate that the delays grow
as $\left(1 - \rho \right)^{-L^2 / 2}$, which is still
faster than the bounds we obtained for fixed back-off probabilities.
In order for the lower bounds for queue-based back-off probabilities
in~\cite{BBvL13} to match the lower bounds in the present paper for fixed
back-off probabilities, the function $p(q)$ should decay as
$q^{- 1/ L}$, but it is not clear whether maximum stability would remain
guaranteed in that case.

\section*{Acknowledgment}
This work was financially supported by a TOP Grant from The Netherlands Organization for Scientific Research (NWO) and an ERC Starting Grant. The authors gratefully acknowledge helpful discussions with Patrick Thiran, Reint den Toonder and Peter van de Ven during the early stages of this work.

\bibliographystyle{elsarticle-num}
\bibliography{arxiv}
\end{document}